\documentclass[10pt]{article}
\usepackage{amsmath,amssymb}
\usepackage{enumerate}
\usepackage[dvipsnames]{xcolor}

\allowdisplaybreaks[1]

\newtheorem{theorem}{Theorem}[section]

\newtheorem{lemma}[theorem]{Lemma}
\newtheorem{proposition}[theorem]{Proposition}

\newtheorem{definition}[theorem]{Definition}
\newtheorem{remark}[theorem]{Remark}

\newtheorem{question}[theorem]{Question}

\newcommand{\Z}{\mathbb{Z}}

\renewcommand{\ker}{\operatorname{Ker}}
\newcommand{\id}{\operatorname{id}}

\newcommand{\Sym}{\operatorname{Sym}}
\newcommand{\aut}{\operatorname{Aut}}

\newcommand{\soc}{\operatorname{Soc}}

\newcommand{\Aut}{\operatorname{Aut}}

\newcommand{\Ret}{\operatorname{Ret}}
\newcommand{\gr}{\operatorname{gr}}

\newenvironment{proof}{\par\noindent{ Proof.}}{$\qed$\par\bigskip}
\newcommand{\qed}{\enspace\vrule  height6pt  width4pt  depth2pt}
\usepackage{color}

\begin{document}
\title{New simple solutions of the Yang--Baxter equation and their permutation
groups}
\author{F. Ced\'o \and J. Okni\'{n}ski}
\date{}

\maketitle

\vspace{30pt}
 \noindent \begin{tabular}{llllllll}
  F. Ced\'o && J. Okni\'{n}ski \\
 Departament de Matem\`atiques &&  Institute of
Mathematics \\
 Universitat Aut\`onoma de Barcelona &&    University of Warsaw\\
08193 Bellaterra (Barcelona), Spain    &&  Banacha 2, 02-097 Warsaw, Poland \\
 Ferran.Cedo@uab.cat && okninski@mimuw.edu.pl
\end{tabular}\\

\vspace{30pt} \noindent Keywords: Yang--Baxter equation,
set-theoretic solution,
 indecomposable solution,  simple solution, brace\\


\begin{abstract}
    A new class of indecomposable, irretractable, involutive, non-degen\-erate
    set-theoretic  solutions of the Yang--Baxter equation is constructed.
    This class complements the class of such solutions constructed in \cite{CO22}
    and {together they} generalize the class of solutions described
    in \cite[Theorem 4.7]{CO21}. Necessary and sufficient conditions are
    found in order that these new solutions are simple. For a rich subclass
    of these solutions the structure of their permutation groups,
    considered as left
    braces, is determined. In particular, these results answer a question
    stated in \cite{CO21}.
    In the finite case, all these solutions have square cardinality.
    A new class of finite simple solutions of non-square cardinality
    such that their permutation groups are simple left braces is also
    constructed.
\end{abstract}

\newpage

\section{Introduction}

This work is devoted to the study of non-degenerate involutive
set-theoretic solutions of the Yang-Baxter equation, YBE for
short. The theory has exploded in the recent two decades, to a
large extent due to the fundamental papers \cite{ESS,GIVdB}. In
particular, the key notions of indecomposable solutions and
irretractable solutions were defined in \cite{ESS} and an
important tool, called the permutation group associated to a
solution, was introduced. The role of indecomposable solutions as
the building blocks in the theory became clear and certain
important results have been obtained \cite{Rump1}. A new impetus
came later due to the introduction of a new powerful tool,
algebraic structures called braces, in \cite{R07}, see also
\cite{CedoSurvey}. Several new ideas and methods, have been
recently developed and applied in this context \cite{CJOComm,
CO21, lebed_and_co, rump2020,Smok, Smokt} and many new classes of
solutions have been constructed \cite{CJOprimit, CO23}. But the
challenging problem of describing all indecomposable solutions
seems very difficult and remains wide open \cite{CPR,
Jedl-Pilit-Zam,Jedl-Pilit-Zam2021,Rump2020,SmokSmok}.

The first aim of this paper is to construct new indecomposable,
irretractable, involutive, non-degenerate set-theoretic solutions
$(X,r)$ of the YBE.

One of our main motivations comes from problems concerning the so
called simple solutions, introduced in \cite{V}. Namely, by
\cite[Proposition~4.1]{CO21}, if $(X,r)$ is a simple solution of
the YBE and $|X|>2$, then $(X,r)$ is indecomposable. Moreover, by
\cite[Proposition~4.2]{CO21}, if $|X|$ is not a prime, then
$(X,r)$ is irretractable.  The main result of \cite{CO23}
shows that every indecomposable solution of square-free
cardinality must be retractable. In \cite[Section 4]{CO21}, a
method is given to construct a big family of simple solutions
$(X,r)$ of the YBE of cardinality $n^2$. In \cite[Section
6]{CO21}, an alternative construction is given of some of the
simple solutions constructed in \cite[Section 4]{CO21} using the
asymmetric product of left braces. The following question, stated
in \cite[Question 7.4]{CO21}, is natural: Consider the
indecomposable and irretractable solutions constructed in
\cite[Theorem 4.7]{CO21}. Can they be constructed using the
asymmetric product of left braces?

In \cite{CO22}, a bigger family of indecomposable and
irretractable solutions of the YBE is constructed, that contains
the solutions built in \cite[Theorem 4.7]{CO21} for the value
of the parameter $t=1$. Moreover, a criterion for the simplicity
of these solutions is provided. In \cite[Section 4]{CO22}, it is
proven that these solutions can be constructed using the
asymmetric product of left braces.

In Section \ref{section5} we extend the results in \cite{CO22},
constructing a rich class of indecomposable and irretractable
solutions of the YBE depending on a parameter $t\in\Aut(A)$, where
$A$ is an abelian group, that includes the solutions described in
\cite[Theorem 4.7]{CO21}, for arbitrary invertible $t\in\Z/(n)$.
We also give necessary and sufficient conditions in order that
these solutions are simple.

In Section \ref{section7}, it is proven that if $t-\id$ is
invertible, then the solutions constructed in Section
\ref{section5} can be also obtained using the asymmetric product
of left braces. This allows us to determine the structure of their
permutation groups as left braces.

In Section \ref{section8}, we construct a new family of
indecomposable and irretractable finite solutions of the YBE of
non-square cardinality. We also show that some of them are simple
and their permutation groups are simple as left braces.

\section{Preliminaries} \label{prelim}
Let $X$ be a non-empty set and  let  $r:X\times X \longrightarrow
X\times X$ be a map. For $x,y\in X$ we put $r(x,y) =(\sigma_x (y),
\gamma_y (x))$. Recall that $(X,r)$ is an involutive,
non-degenerate set-theoretic solution of the Yang--Baxter equation
if $r^2=\id$, all the maps $\sigma_x$ and $\gamma_y$ are bijective
maps from $X$ to itself and
  $$r_{12} r_{23} r_{12} =r_{23} r_{12} r_{23},$$
where $r_{12}=r\times \id_X$ and $r_{23}=\id_X\times \ r$ are maps
from $X^3$ to itself. Because $r^{2}=\id$, one easily verifies
that $\gamma_y(x)=\sigma^{-1}_{\sigma_x(y)}(x)$, for all $x,y\in
X$ (see for example \cite[Proposition~1.6]{ESS}).

\bigskip
\noindent {\bf Convention.} Throughout the paper a solution of the
YBE will mean an involutive, non-degenerate, set-theoretic
solution of the Yang--Baxter equation. \\

It is well known, see for example \cite[Proposition
8.2]{CedoSurvey}, that a map $r(x,y) = (\sigma_x (y),
\sigma^{-1}_{\sigma_x(y)}(x))$, defined for a collection of
bijections $\sigma_x, x\in X$, of the set $X$, is a solution of
the YBE if and only if
$\sigma_{x}\sigma_{\sigma_{x}^{-1}(y)}=\sigma_{y}\sigma_{\sigma_{y}^{-1}(x)}$
for all $x,y\in X$ and the maps $\gamma_y\colon X\rightarrow X$
defined by $\gamma_y(x)=\sigma^{-1}_{\sigma_x(y)}(x)$, for all
$x,y\in X$, are bijective. Furthermore, if $X$ is finite, then
$(X,r)$ is a solution of the YBE if and only if
$\sigma_{x}\sigma_{\sigma_{x}^{-1}(y)}=\sigma_{y}\sigma_{\sigma_{y}^{-1}(x)}$
    for all $x,y\in X$.

A left brace is a set $B$ with two binary operations, $+$ and
$\circ$, such that $(B,+)$ is an abelian group (the additive group
of $B$), $(B,\circ)$ is a group (the multiplicative group of $B$),
and for every $a,b,c\in B$,
 \begin{eqnarray} \label{braceeq}
  a\circ (b+c)+a&=&a\circ b+a\circ c.
 \end{eqnarray}
In any left brace $B$  the neutral elements $0,1$ for the
operations $+$ and $\circ$ coincide. Moreover, there is an action
$\lambda\colon (B,\circ)\longrightarrow \aut(B,+)$, called the
lambda map of $B$, defined by $\lambda(a)=\lambda_a$ and
$\lambda_{a}(b)=-a+a\circ b$, for $a,b\in B$. We shall write
$a\circ b=ab$ and $a^{-1}$ will denote the inverse of $a$ for the
operation $\circ$, for all $a,b\in B$. A trivial brace is a left
brace $B$ such that $ab=a+b$, for all $a,b\in B$, i.e. all
$\lambda_a=\id$. The socle of a left brace $B$ is
$$\soc(B)=\{ a\in B\mid ab=a+b, \mbox{ for all
}b\in B \}.$$ Note that $\soc(B)=\ker(\lambda)$, and thus it is a
normal subgroup of the multiplicative group of $B$. The solution of
the YBE associated to a left brace $B$ is $(B,r_B)$, where
$r_B(a,b)=(\lambda_a(b),\lambda_{\lambda_a(b)}^{-1}(a))$, for all
$a,b\in B$ (see \cite[Lemma~2]{CJOComm}).

A left ideal of a left brace $B$ is a subgroup $L$ of the additive
group of $B$ such that $\lambda_a(b)\in L$, for all $b\in L$ and
all $a\in B$. An ideal of a left brace $B$ is a normal subgroup
$I$ of the multiplicative group of $B$ such that $\lambda_a(b)\in
I$, for all $b\in I$ and all $a\in B$. Note that
\begin{eqnarray}\label{addmult1}
ab^{-1}&=&a-\lambda_{ab^{-1}}(b)
\end{eqnarray}
 for all $a,b\in B$, and
    \begin{eqnarray} \label{addmult2}
     &&a-b=a+\lambda_{b}(b^{-1})= a\lambda_{a^{-1}}(\lambda_b(b^{-1}))= a\lambda_{a^{-1}b}(b^{-1}),
     \end{eqnarray}
for all $a,b\in B$. Hence, every left ideal $L$ of $B$ also is a
subgroup of the multiplicative group of $B$, and every  ideal $I$
of a left brace $B$ also is a subgroup of the additive group of
$B$. For example,  every Sylow subgroup of the additive group of a
finite left brace $B$ is a left ideal of $B$. Consequently, if a
Sylow subgroup of the multiplicative group $(B,\circ)$ is normal,
then it must be an ideal of the finite left brace $B$. It is also
known that $\soc(B)$ is an ideal of the left brace $B$ (see
\cite[Proposition~7]{R07}). Note that, for every ideal $I$ of $B$,
$B/I$ inherits a natural left brace structure. A nonzero left
brace $B$ is simple if $\{0\}$ and $B$ are the only ideals of $B$.

A homomorphism of left braces is a map $f\colon B_1\longrightarrow
B_2$, where $B_1,B_2$ are left braces, such that $f(a b)=f(a)
f(b)$ and $f(a+b)=f(a)+f(b)$, for all $a,b\in B_1$. Note that the
kernel $\ker(f)$ of a homomorphism of left braces $f\colon
B_1\longrightarrow B_2$ is an ideal of $B_1$.

Recall that if $(X,r)$ is a solution of the YBE, with
$r(x,y)=(\sigma_x(y),\gamma_y(x))$, then its structure group
$G(X,r)=\gr(x\in X\mid xy=\sigma_x(y)\gamma_y(x)$ for all $x,y\in
X)$ has a natural structure of a left brace such that
$\lambda_x(y)=\sigma_x(y)$, for all $x,y\in X$. The additive group
of $G(X,r)$ is the free abelian group with basis $X$. The
permutation group $\mathcal{G}(X,r)=\gr(\sigma_x\mid x\in X)$ of
$(X,r)$ is a subgroup of the symmetric group $\Sym_X$ on $X$.  The
map $x\mapsto \sigma_x$, from $X$ to $\mathcal{G}(X,r)$ extends to
a group homomorphism $\phi: G(X,r)\longrightarrow
\mathcal{G}(X,r)$ and $\ker(\phi)=\soc(G(X,r))$. Hence there is a
unique structure of a left brace on $\mathcal{G}(X,r)$ such that
$\phi$ is a homomorphism of left braces; this is the natural
structure of a left brace on $\mathcal{G}(X,r)$. In particular,
\begin{equation*}
    \sigma_{\sigma_x(y)}=\phi(\sigma_x(y))=\phi(\lambda_x(y))=\lambda_{\phi(x)}(\phi(y))=\lambda_{\sigma_x}(\sigma_y)
\end{equation*}
for all $x,y\in X$. An easy consequence is \cite[Lemma
2.1]{CJOprimit}, which says that
\begin{equation}\label{Lemma 2.1}
    \lambda_{g}(\sigma_y)=\sigma_{g(y)}
\end{equation}
for all $g\in\mathcal{G}(X,r)$ and all $y\in X$.

Let $(X,r)$ and $(Y,s)$ be solutions of the YBE. We write
$r(x,y)=(\sigma_x(y),\gamma_y(x))$ and
$s(t,z)=(\sigma'_t(z),\gamma'_z(t))$, for all $x,y\in X$ and $t,z\in
Y$. A homomorphism of solutions $f\colon (X,r)\longrightarrow (Y,s)$
is a map $f\colon X\longrightarrow Y$ such that
$f(\sigma_x(y))=\sigma'_{f(x)}(f(y))$ and
$f(\gamma_y(x))=\gamma'_{f(y)}(f(x))$, for all $x,y\in X$. Since
$\gamma_y(x)=\sigma^{-1}_{\sigma_x(y)}(x)$ and
$\gamma'_z(t)=(\sigma')^{-1}_{\sigma'_t(z)}(t)$, it is clear that
$f$ is a homomorphism of solutions if and only if
$f(\sigma_x(y))=\sigma'_{f(x)}(f(y))$, for all $x,y\in X$.

Note  that, by de defining relations of the structure group, every
homomorphism of solutions $f\colon (X,r)\longrightarrow (Y,s)$
extends to a unique homomorphism of groups $f\colon
G(X,r)\longrightarrow G(Y,s)$ that we also denote by $f$, and  it
is easily checked that it also is a homomorphism of left braces.
If $f$ is surjective, then $f(\soc(G(X,r)))\subseteq \soc(G(Y,s))$
and thus $f$ induces a homomorphism of left braces $\bar f\colon
\mathcal{G}(X,r)\longrightarrow\mathcal{G}(Y,s)$.

In \cite{ESS}, Etingof, Schedler and Soloviev introduced the
retract relation on solutions  $(X,r)$ of the YBE. This is the
binary relation $\sim$ on $X$ defined by $x\sim y$ if and only if
$\sigma_x=\sigma_y$. Then, $\sim$ is an equivalence relation and
$r$ induces a solution $\overline{r}$ on the set
$\overline{X}=X/{\sim}$. The retract of the solution $(X,r)$ is
$\Ret(X,r)=(\overline{X},\overline{r})$. Note that the natural map
$f\colon X\longrightarrow \overline{X}:x\mapsto \bar x$ is an
epimorphism of solutions from $(X,r)$ onto $\Ret(X,r)$.

Recall that a solution $(X,r)$ is  said to be irretractable if
$\sigma_x\neq \sigma_y$ for all distinct elements $x,y\in X$, that
is $(X,r)=\Ret(X,r)$; otherwise the solution $(X,r)$ is
retractable.  Define $\Ret^1(X,r)=\Ret(X,r)$ and, for every
integer $n>1$, $\Ret^{n}(X,r)=\Ret(\Ret^{n-1}(X,r))$. A solution
$(X,r)$ is said to be a multipermutation solution if there exists
a positive integer $n$ such that $\Ret^n(X,r)$ has cardinality
$1$.

Let $(X,r)$ be a solution of the YBE. We say that $(X,r)$ is
indecomposable if $\mathcal{G}(X,r)$ acts transitively on $X$.
We write $r(x,y)=(\sigma_x(y),\gamma_y(x))$.

\begin{definition}
    A solution $(X,r)$ of the YBE is simple if $|X|>1$ and for every
    epimorphism $f:(X,r) \longrightarrow (Y,s)$ of solutions either $f$
    is an isomorphism or $|Y|=1$.
\end{definition}

In this context, the following result (Proposition~4.1 in
\cite{CO21} and Lemma~3.4 in \cite{CO22}) is crucial.

\begin{lemma}  \label{simple indec}
Assume that $(X,r)$ is a simple solution of the YBE. Then it is
indecomposable if $|X|>2$ and it is irretractable if $|X|$ is not
a prime number.
\end{lemma}

\section{New simple solutions of cardinality $n^2$}\label{section5}

In this section we start with a generalization of \cite[Theorem
4.7]{CO21}.

\begin{theorem}\label{newsol}
    Let $A$ be an abelian group. Let $t\in\Aut(A)$. Let $(j_a)_{a\in A}$ be a family of elements of $A$ such that $j_a=j_{-a}$ and
    \begin{equation}\label{jt}
        j_{t^s(a)}-j_0=t^{s}(j_a-j_0),
    \end{equation}
    for all $a\in A$ and $s\in \Z$. Let $r\colon A^2\times A^2\longrightarrow A^2\times A^2$ be the map defined by
    $$r((a_1,a_2),(c_1,c_2))=(\sigma_{(a_1,a_2)}(c_1,c_2),\sigma^{-1}_{\sigma_{(a_1,a_2)}(c_1,c_2)}(a_1,a_2)),$$
    where $\sigma_{(a_1,a_2)}(c_1,c_2)=(t(c_1)+a_2,t(c_2-j_{t(c_1)+a_2-a_1}))$, for all $a_1,a_2,c_1,c_2\in A$.
    Then $(A^2,r)$ is a solution of the YBE.
\end{theorem}

\begin{proof}
    Note that $\sigma_{(a_1,a_2)}\in \Sym_{A^2}$ and
    \begin{equation}\label{sigmainverse}
    \sigma^{-1}_{(a_1,a_2)}(c_1,c_2)=(t^{-1}(c_1-a_2),t^{-1}(c_2)+j_{c_1-a_1}),
    \end{equation}
    for all $a_1,a_2,c_1,c_2\in A$. It is known from \cite[Proposition 2]{CJOComm} that $(A^2,r)$ is a solution of the YBE if and only if
    \begin{equation}\label{sigmacond}
        \sigma_{(a_1,a_2)}\sigma_{\sigma^{-1}_{(a_1,a_2)}(c_1,c_2)}=\sigma_{(c_1,c_2)}\sigma_{\sigma^{-1}_{(c_1,c_2)}(a_1,a_2)},
    \end{equation}
    for all $a_1,a_2,c_1,c_2\in A$, and for every $a_1,a_2\in A$, the map $\gamma_{(a_1,a_2)}\colon A\longrightarrow A$ defined by
    $\gamma_{(a_1,a_2)}(c_1,c_2)=\sigma^{-1}_{\sigma_{(c_1,c_2)}(a_1,a_2)}(c_1,c_2)$, for all $c_1,c_2\in A$, is bijective. We have that
    \begin{eqnarray*}
        \lefteqn{\sigma_{(a_1,a_2)}\sigma_{\sigma^{-1}_{(a_1,a_2)}(c_1,c_2)}(x,y)}\\
        &=&\sigma_{(a_1,a_2)}\sigma_{(t^{-1}(c_1-a_2),t^{-1}(c_2)+j_{c_1-a_1})}(x,y)\\
        &=&\sigma_{(a_1,a_2)}(t(x)+t^{-1}(c_2)+j_{c_1-a_1}, t(y-j_{t(x)+t^{-1}(c_2)+j_{c_1-a_1}-t^{-1}(c_1-a_2)}))\\
        &=&(t(t(x)+t^{-1}(c_2)+j_{c_1-a_1})+a_2,\\
        &&\qquad t(t(y-j_{t(x)+t^{-1}(c_2)+j_{c_1-a_1}-t^{-1}(c_1-a_2)})-j_{t(t(x)+t^{-1}(c_2)+j_{c_1-a_1})+a_2-a_1}))\\
        &=&(t^2(x)+c_2+t(j_{c_1-a_1})+a_2,\\
        &&\qquad t^2(y)-t^2(j_{t(x)+t^{-1}(c_2)+j_{c_1-a_1}-t^{-1}(c_1-a_2)})-t(j_{t^2(x)+c_2+t(j_{c_1-a_1})+a_2-a_1}))
    \end{eqnarray*}
    and similarly
    \begin{eqnarray*}
        \lefteqn{\sigma_{(c_1,c_2)}\sigma_{\sigma^{-1}_{(c_1,c_2)}(a_1,a_2)}(x,y)}\\
        &=&(t^2(x)+a_2+t(j_{a_1-c_1})+c_2,\\
        &&\qquad t^2(y)-t^2(j_{t(x)+t^{-1}(a_2)+j_{a_1-c_1}-t^{-1}(a_1-c_2)})-t(j_{t^2(x)+a_2+t(j_{a_1-c_1})+c_2-c_1})),
    \end{eqnarray*}
    for all $a_1,a_2,c_1,c_2,x,y\in A$. Since $j_a=j_{-a}$, we have that
    $$t^2(x)+a_2+t(j_{a_1-c_1})+c_2=t^2(x)+c_2+t(j_{c_1-a_1})+a_2$$
    and using (\ref{jt})  (twice, in order to get the first equality) we get
    \begin{eqnarray*}
    \lefteqn{t^2(y)-t^2(j_{t(x)+t^{-1}(a_2)+j_{a_1-c_1}-t^{-1}(a_1-c_2)})-t(j_{t^2(x)+a_2+t(j_{a_1-c_1})+c_2-c_1})}\\
    &=&t^2(y)-t^2(j_0)+t(j_0)-t(j_{t^2(x)+a_2+t(j_{c_1-a_1})-a_1+c_2})\\
    &&\quad -t(j_0)+t^2(j_0)-t^2(j_{t(x)+t^{-1}(a_2)+j_{c_1-a_1}+t^{-1}(c_2)-t^{-1}(c_1)})\\
    &=&t^2(y)-t^2(j_{t(x)+t^{-1}(c_2)+j_{c_1-a_1}-t^{-1}(c_1-a_2)})-t(j_{t^2(x)+c_2+t(j_{c_1-a_1})+a_2-a_1}),
    \end{eqnarray*}
     for all $a_1,a_2,c_1,c_2,x,y\in A$. Hence (\ref{sigmacond}) is satisfied.
     Let $T\colon A\longrightarrow A$ be the map defined by $T(a_1,a_2)=\sigma^{-1}_{(a_1,a_2)}(a_1,a_2)$, for all $a_1,a_2\in A$.
     Note that by the inverse of (\ref{sigmacond}) applied to the pair
     of elements $\sigma_{(c_1,c_2)}(a_1,a_2), (c_1,c_2)$, and by the definition of
     $\gamma_{(a_1,a_2)}$,
     \begin{align*}
        T(\gamma_{(a_1,a_2)}(c_1,c_2))
        =&\sigma^{-1}_{\sigma^{-1}_{\sigma_{(c_1,c_2)}(a_1,a_2)}(c_1,c_2)}(\sigma^{-1}_{\sigma_{(c_1,c_2)}(a_1,a_2)}(c_1,c_2))\\
        =&\sigma^{-1}_{\sigma^{-1}_{(c_1,c_2)}(\sigma_{(c_1,c_2)}(a_1,a_2))}(\sigma^{-1}_{(c_1,c_2)}(c_1,c_2))\\
        =&\sigma^{-1}_{(a_1,a_2)}(T(c_1,c_2)),
     \end{align*}
     for all $a_1,a_2,c_1,c_2$. Hence to prove that $\gamma_{(a_1,a_2)}$ is bijective it is enough to prove that $T$ is bijective.
     Since $T(a_1,a_2)=(t^{-1}(a_1-a_2), t^{-1}(a_2)+j_0)$, it is easy to see that $T$ is bijective and
     $$T^{-1}(a_1,a_2)=(t(a_1+a_2-j_0),t(a_2-j_0)),$$
     for all $a_1,a_2\in A$. Therefore $(A^2,r)$ is a solution of the YBE.
\end{proof}

\begin{proposition}\label{newindecomposable}
    With the same notation and hypothesis as in Theorem \ref{newsol}.
    Let $W=\gr(j_a-j_c\mid a,c\in A)$. Then the orbit of $(0,0)$ by the action of
    $\mathcal{G}(A^2,r)$ is $A\times C$, where
    \begin{align*}C=&\{ w+j_0+t^{-1}(j_0)+\dots +t^{-n}(j_0)\mid w\in W \mbox{ and }n\mbox{ is a non-negative integer} \}\\
        &\cup
    \{ w-t(j_0)-\dots -t^{n}(j_0)\mid w\in W \mbox{ and }n\mbox{ is a non-negative integer} \}.
    \end{align*}
    In particular, if $W=A$, then the solution $(A^2,r)$ of the YBE is indecomposable.
\end{proposition}

\begin{proof}
     By (\ref{jt}), we have that $t^s(j_a-j_c)=j_{t^s(a)}-j_{t^s(c)}$, for all $a,c\in A$
     and all $s\in\Z$. Hence $t^s(W)=W$, for all $s\in\Z$.
     Let $n$ be a non-negative integer. First we prove that $g(A\times C)=A\times C$,
     for all $g\in\mathcal{G}(A^2,r)$. In fact, we have that
    \begin{eqnarray}
        \label{orbit00}\lefteqn{\sigma_{(a_1,a_2)}(c,w+j_0+t^{-1}(j_0)+\dots+t^{-n}(j_0))}\\
        &=&(t(c)+a_2, t(w+j_0+t^{-1}(j_0)+\dots+t^{-n}(j_0)-j_{t(c)+a_2-a_1}))\notag\\
        &=&(t(c)+a_2,t(w+j_0-j_{t(c)+a_2-a_1})+j_0+\dots +t^{-n+1}(j_0))\in A\times C,\notag
    \end{eqnarray}
    \begin{eqnarray*}\lefteqn{\sigma^{-1}_{(a_1,a_2)}(c,w+j_0+t^{-1}(j_0)+\dots+t^{-n}(j_0))}\\
        &=&(t^{-1}(c-a_2),t^{-1}(w+j_0+t^{-1}(j_0)+\dots+t^{-n}(j_0))+j_{c-a_1})\\
        &=&(t^{-1}(c-a_2),t^{-1}(w)+t^{-1}(j_0)+\dots+t^{-n-1}(j_0)+j_{c-a_1}-j_0+j_0)\in A\times C,
    \end{eqnarray*}
    \begin{eqnarray}\label{orbit002}\lefteqn{\sigma_{(a_1,a_2)}(c,w-t(j_0)-\dots -t^{n}(j_0))}\\
        &=&(t(c)+a_2, t(w-t(j_0)-\dots -t^{n}(j_0)-j_{t(c)+a_2-a_1}))\notag\\
        &=&(t(c)+a_2,t(w+j_0-j_{t(c)+a_2-a_1})-t(j_0)-\dots -t^{n+1}(j_0))\in A\times C\notag
    \end{eqnarray}
    and
    \begin{eqnarray*}\lefteqn{\sigma^{-1}_{(a_1,a_2)}(c,w-t(j_0)-\dots -t^{n}(j_0))}\\
        &=&(t^{-1}(c-a_2),t^{-1}(w-t(j_0)-\dots -t^{n}(j_0))+j_{c-a_1})\\
        &=&(t^{-1}(c-a_2),t^{-1}(w)-t(j_0)-\dots -t^{n-1}(j_0)+j_{c-a_1}-j_0)\in A\times C,
    \end{eqnarray*}
    for all $a_1,a_2,c\in A$ and all $w\in W$. Hence $g(A\times C)=A\times C$, for all
    $g\in \mathcal{G}(A^2,r)$.

    Let $w\in W$. There exist  a positive integer $k$ and $a_{i},c_{i}\in A$, for $0\leq i\leq k$, such that
    $$w=\sum_{i=0}^k(j_{a_i}-j_{c_i})$$
    We shall prove that $(0,w)$ is in the orbit of $(0,0)$ by induction on $k$. For $k=0$, $w=j_{a_0}-j_{c_0}$.
    We may assume that $j_{a_0}\neq j_{c_0}$. Note that
    $$\sigma^{-1}_{(t^{-1}(a_0),0)}(0,0)=(0,j_{-t^{-1}(a_0)})=(0,j_{t^{-1}(a_0)}),$$
    and by (\ref{jt}),
    $$\sigma_{(t^{-1}(c_0),0)}(0,j_{t^{-1}(a_0)})=(0,t(j_{t^{-1}(a_0)}-j_{-t^{-1}(c_0)}))=(0,j_{a_0}-j_{c_0}).$$
    Hence $(0,w)$ is in the orbit of $(0,0)$ for $k=0$. Suppose that $k>0$ and that $(0,\sum_{i=0}^{k-1}(j_{a_i}-j_{c_i}))$
    is in the orbit of $(0,0)$. We have that
        \begin{align*}\sigma^{-1}_{(t^{-1}(a_k),0)}(0,\sum_{i=0}^{k-1}(j_{a_i}-j_{c_i}))
            =&(0,t^{-1}(\sum_{i=0}^{k-1}(j_{a_i}-j_{c_i}))+j_{-t^{-1}(a_k)})\\
            =&(0,t^{-1}(\sum_{i=0}^{k-1}(j_{a_i}-j_{c_i}))+j_{t^{-1}(a_k)}),
            \end{align*}
    and by (\ref{jt}),
    \begin{eqnarray*}\lefteqn{\sigma_{(t^{-1}(c_k),0)}(0,t^{-1}(\sum_{i=0}^{k-1}(j_{a_i}-j_{c_i}))+j_{t^{-1}(a_k)})}\\
        &=&(0,\sum_{i=0}^{k-1}(j_{a_i}-j_{c_i})+t(j_{t^{-1}(a_k)}-j_{-t^{-1}(c_k)}))\\
        &=&(0,\sum_{i=0}^{k-1}(j_{a_i}-j_{c_i})+j_{a_k}-j_{c_k})=(0,w).
    \end{eqnarray*}
    Hence $(0,w)$ is in the orbit of $(0,0)$. Therefore, by induction, $(0,w)$ is in the orbit of $(0,0)$ for all $w\in W$.
    Since
    $$\sigma_{(0,a)}\sigma^{-1}_{(0,c)}(0,w)=\sigma_{(0,a)}(-t^{-1}(c),t^{-1}(w)+j_{0})=(a-c,w+t(j_0-j_{-t^{-1}(c)})),$$
    for all $a,c\in A$ and all $w\in W$, we get that $(c,w)$ is in the orbit of $(0,0)$ for all $c\in A$ and all $w\in W$.
    By (\ref{orbit00}) and (\ref{orbit002}), one can prove by induction on $n$ that $(a,u)$ is in the orbit of $(0,0)$
    for all $a\in A$ and all $u\in C$. Hence the orbit of $(0,0)$ is $A\times C$ and the result follows.
\end{proof}

\begin{proposition}\label{newirretractable}
With the same notation and hypothesis as in Theorem \ref{newsol}.
The solution $(A^2,r)$ of the YBE is irretractable if and only if
for every nonzero $a\in A$ there exists $c\in A$ such that
$j_{a+c}\neq j_c$.
\end{proposition}

\begin{proof}
    Suppose that $\sigma_{(a_1,a_2)}=\sigma_{(c_1,c_2)}$. Since
    $$\sigma_{(a_1,a_2)}(c,0)=(t(c)+a_2,-t(j_{t(c)+a_2-a_1})),$$
    we have that $a_2=c_2$ and $j_{t(c)+a_2-a_1}=j_{t(c)+a_2-c_1}$, for all $c\in A$.

    Conversely, suppose that there exist a nonzero $a\in A$ such that $j_{a+c}=j_c$ for all $c\in A$.
    Then
    $$\sigma_{(0,0)}(a_1,a_2)=(t(a_1),t(a_2-j_{t(a_1)}))=(t(a_1),t(a_2-j_{a+t(a_1)}))=\sigma_{(-a,0)}(a_1,a_2),$$
    for all $a_1,a_2\in A$. Hence $\sigma_{(0,0)}=\sigma_{(-a,0)}$ and the result follows.
\end{proof}

Consider the solution $(A^2,r)$ of the YBE described in Theorem
\ref{newsol}. Let $a\in A$ be a nonzero element. Let
$V_{a,1}=\gr(j_c-j_{c+t^z(a)}\mid c\in A,\; z\in\Z)$. For every
$i>1$, define $V_{a,i}=V_{a,i-1}+\gr(j_c-j_{c+v}\mid c\in A,\;
v\in V_{a,i-1})$. Let $V_a=\sum_{i=1}^{\infty}V_{a,i}$. Note that
$V_a=\bigcup_{i=1}^{\infty}V_{a,i}$ is a subgroup of $A$. By
(\ref{jt}),
$$t(j_c-j_{c+t^z(a)})=j_{t(c)}-j_{t(c)+t^{z+1}(a)}\in V_{a,1}$$
and
$$t^{-1}(j_c-j_{c+t^z(a)})=j_{t^{-1}(c)}-j_{t^{-1}(c)+t^{z-1}(a)}\in V_{a,1},$$
for all $c\in A$ and all $z\in\Z$. Hence $t(V_{a,1})=V_{a,1}$. By
induction on $i$, one can prove that $t(V_{a,i})=V_{a,i}$, for
every positive integer $i$. Hence $t(V_a)=V_a$.

\begin{theorem} \label{simple_necessary}
    With the same notation and hypothesis as in Theorem \ref{newsol}.
    Suppose that $A\neq\{ 0\}$. If the solution $(A^2,r)$ is simple, then $V_a=A$,
    for every nonzero element $a\in A$.
\end{theorem}

\begin{proof}
    Suppose that $(A^2,r)$ is simple. Let $a\in A$ be a nonzero element.
    Let $\pi\colon A\longrightarrow A/V_a$ be the natural map. Note that
    \begin{eqnarray*}
        \lefteqn{\sigma_{(a_1+v_1,a_2+v_2)}(c_1+v_3,c_2+v_4)}\\
        &=&(t(c_1+v_3)+a_2+v_2,t(c_2+v_4-j_{t(c_1+v_3)+a_2+v_2-a_1-v_1})\\
        &=&(t(c_1)+a_2+t(v_3)+v_2,\\
        &&\qquad t(c_2-j_{t(c_1)+a_2-a_1})+t(j_{t(c_1)+a_2-a_1}-j_{t(c_1)+a_2-a_1+t(v_3)+v_2-v_1})+t(v_4)),
    \end{eqnarray*}
    for all $a_1,a_2,c_1,c_2\in A$ and all $v_1,v_2,v_3,v_4\in V_a$.
    Since $$j_{t(c_1)+a_2-a_1}-j_{t(c_1)+a_2-a_1+t(v_3)+v_2-v_1}\in V_a$$
    and $t(V_a)=V_a$, the map
    $$\sigma'_{(\pi(a_1),\pi(a_2))}\colon (A/V_a)^2\longrightarrow (A/V_a)^2,$$
    defined by
    $$\sigma'_{(\pi(a_1),\pi(a_2))}(\pi(c_1),\pi(c_2))=(\pi(t(c_1)+a_2),\pi(t(c_2-j_{t(c_1)+a_2-a_1}))),$$
    is well-defined. Similarly one can see that the map $\sigma''_{(\pi(a_1),\pi(a_2))}\colon (A/V_a)^2\longrightarrow (A/V_a)^2$,
    defined by
    $$\sigma''_{(\pi(a_1),\pi(a_2))}(\pi(c_1),\pi(c_2))=(\pi(t^{-1}(c_1-a_2)),\pi(t^{-1}(c_2)+j_{c_1-a_1})),$$
    is well-defined and $\sigma''_{(\pi(a_1),\pi(a_2))}=(\sigma'_{(\pi(a_1),\pi(a_2))})^{-1}$.
    Also one can check that the maps $T'\colon (A/V_a)^2\longrightarrow (A/V_a)^2$ and $T''\colon (A/V_a)^2\longrightarrow (A/V_a)^2$,
    defined by
    $$T'(\pi(a_1),\pi(a_2))=\sigma''_{(\pi(a_1),\pi(a_2)}(\pi(a_1),\pi(a_2))$$
    and
    $$T''(\pi(a_1),\pi(a_2))=(\pi(t(a_1+a_2-j_0)),\pi(t(a_2-j_0))),$$
    are well-defined and $T''=(T')^{-1}$. Let $r'\colon (A/V_a)^{2}\times (A/V_a)^2\longrightarrow (A/V_a)^{2}\times (A/V_a)^2$
    be the map defined by
    \begin{eqnarray*}
        \lefteqn{r'((\pi(a_1),\pi(a_2)),(\pi(c_1),\pi(c_2)))}\\
        &=&\left( \sigma'_{(\pi(a_1),\pi(a_2))}(\pi(c_1),\pi(c_2)),\left(\sigma'_{\sigma'_{(\pi(a_1),\pi(a_2))}(\pi(c_1),\pi(c_2))}\right)^{-1}(\pi(a_1),\pi(a_2)) \right),
    \end{eqnarray*}
    for all $a_1,a_2,c_1,c_2\in A$.
    Then $((A/V_a)^2,r')$ is a solution of the YBE and the map $f\colon A^2\longrightarrow (A/V_a)^2$,
    defined by $f(a_1,a_2)=(\pi(a_1),\pi(a_2))$ is an epimorphism of solutions.
    Since $(A^2,r)$ is simple, by Lemma \ref{simple indec}, $(A^2,r)$ is irretractable.
    By Proposition \ref{newirretractable}, $V_{a,1}\neq\{ 0\}$. Hence, since $(A^2,r)$ is simple,
    we have that $V_a=A$, and the result follows.
\end{proof}

The following result  is a partial converse of
Theorem~\ref{simple_necessary} and it generalizes \cite[Theorem
3.6]{CO22}.

\begin{theorem}\label{finiteodd}
    With the same notation and hypothesis as in Theorem \ref{newsol}.
    If $V_a=A$, for every nonzero element $a\in A$, then the solution $(A^2,r)$ is indecomposable
    and irretractable. If moreover $A\neq\{ 0\}$ is finite and either $|A|$ is odd or $t$
    has odd order, then the solution $(A^2,r)$ is simple.
\end{theorem}

\begin{proof}
    Assume that $V_a=A$ for every nonzero element $a\in A$. In
    particular, $\gr ( j_a\mid a\in A)=A$, and thus $(A^2,r)$
    is indecomposable by Proposition \ref{newindecomposable}. Furthermore
    $V_{a,1}=\gr ( j_{c}-j_{c+t^z(a)}\mid c\in A,\; z\in\Z)\neq\{0\}$ for
    every nonzero element $a\in A$. Suppose that there exists a nonzero element $a\in A$ such that $j_c-j_{c+a}=0$,
    for all $c\in A$. Then, by (\ref{jt}),
    $$0=t^z(j_c-j_{c+a})=j_{t^z(c)}-j_{t^z(c)+t^z(a)},$$
    for all $c\in A$ and all $z\in\Z$. Hence $V_{a,1}=\{ 0\}$, a contradiction.
    Therefore for every nonzero element $a\in A$ there exists $c\in A$ such that $j_c-j_{c+a}\neq 0$.
     By Proposition \ref{newirretractable},$(A^2,r)$ is
    irretractable.

    Suppose now that $A\neq\{ 0\}$ is finite. Let $f\colon
    (A^2,r)\longrightarrow (Y,s)$ be an epimorphism of solutions which
    is not an isomorphism. Since $(A^2,r)$ is indecomposable, we have
    that $(Y,s)$ also is indecomposable. By \cite[Lemma 1]{CCP},
    $|f^{-1}(y)|=|f^{-1}(y')|$ for all $y,y'\in Y$.  We write
    $s(y_1,y_2)=(\sigma'_{y_1}(y_2),(\sigma'_{\sigma'_{y_1}(y_2)})^{-1}(y_1))$,
    for all $y_1,y_2\in Y$.

    We will frequently use the following consequence of the definition of a
    homomorphism of solutions:
    $$f(\sigma_{(u,v)}(x,y)) = f(\sigma_{(u',v')}(x',y'))$$
    whenever $f(u,v)=f(u',v')$  and  $f(x,y)=f(x',y')$.

    Note also that
    if $f(u,v)=f(u',v')$  and $f(x,y)=f(x',y')$, then
    \begin{eqnarray*}\sigma'_{f(u,v)}f(\sigma_{(u,v)}^{-1}(x,y))&=&f(\sigma_{(u,v)}(\sigma_{(u,v)}^{-1}(x,y)))\\
        &=&f(x,y)= f(x',y')\\
        &=&f(\sigma_{(u',v')}(\sigma_{(u',v')}^{-1}(x',y')))\\
        &=&\sigma'_{f(u',v')}f(\sigma_{(u',v')}^{-1}(x',y')),
    \end{eqnarray*}
    and since $\sigma'_{f(u,v)}=\sigma'_{f(u',v')}$ is bijective, we get that
    \begin{eqnarray}\label{inverse} f(\sigma_{(u,v)}^{-1}(x,y)) = f(\sigma_{(u',v')}^{-1}(x',y')).
    \end{eqnarray}

    Suppose first that there exist $a_1,a_2,c\in A$ such that
    $f(a_1,c)=f(a_2,c)$ and $a_1\neq a_2$. Let $a=a_2-a_1$. We have
    that
    \begin{eqnarray*}f(t(u)+c,t(v-j_{t(u)+c-a_1}))&=&f(\sigma_{(a_1,c)}(u,v))=f(\sigma_{(a_2,c)}(u,v))\\
        &=&f(t(u)+c,t(v-j_{t(u)+c-a_2})),\end{eqnarray*}
    for all $u,v\in A$.  For $x=t(u)+c$ and $y=t(v-j_{t(u)+c-a_1})$, one gets that
        $$f(x,y)=f(x,y+t(j_{x-a_1}-j_{x-a_2})).$$
    Thus, it also follows that
    $$f(u,v)=f(u,v+t(j_{u-a_1}-j_{u-a_2})),$$
    for all $u,v\in A$.
    Hence, we also get
    \begin{eqnarray*}\lefteqn{f(t(x)+v,t(y-j_{t(x)+v}))}\\
        &=&f(\sigma_{(0,v)}(x,y))\\
        &=&f(\sigma_{(0,v)}(x,y+t(j_{x-a_1}-j_{x-a_2})))\\
        &=&f(t(x)+v,t(y+t(j_{x-a_1}-j_{x-a_2})-j_{t(x)+v})),
    \end{eqnarray*}
    for all $x,y,v\in A$. For $x'=t(x)+v$ and $y'=t(y-j_{t(x)+v})$, this yields
    $$f(x',y')=f(x',y'+t^2(j_{t^{-1}(x'-v)-a_1}-j_{t^{-1}(x'-v)-a_2})).$$
     Thus
    \begin{eqnarray*}f(x,y)&=&f(x,y+t^2(j_{t^{-1}(x-v)-a_1}-j_{t^{-1}(x-v)-a_2})),\end{eqnarray*}
    for all $x,y,v\in A$. Hence by (\ref{jt}), for
    $v=x-t(w)$,
    $$f(x,y)=f(x,y+t^2(j_{w-a_1}-j_{w-a_2}))=f(x,y+(j_{t^2(w-a_1)}-j_{t^2(w-a_2)})),$$
    for all $x,y,w\in A$. Recall that $a=a_1-a_2$. For $c=t^2(w-a_1)$, we have that
    $$f(x,y)=f(x,y+(j_{c}-j_{c+t^2(a_1-a_2)}))=f(x,y+(j_c-j_{c+t^2(a)})),$$
    for all $x,y,c\in A$.
       We shall prove by induction on $i$ that
    \begin{equation}\label{ti}f(x,y)=f(x,y+(j_{c}-j_{c+t^{i}(a)})),
        \end{equation}
     all $x,y,c\in A$ and all $i\geq 2$.
    We have proved (\ref{ti}) for $i=2$. Suppose that (\ref{ti}) holds for some $i\geq 2$.
     Now we have that
    \begin{eqnarray*}\lefteqn{f(t(x),t(y-j_{t(x)}))}\\
        &=&f(\sigma_{(0,0)}(x,y))\\
        &=&f(\sigma_{(0,0)}(x,y+(j_{c}-j_{c+t^{i}(a)})))\\
        &=&f(t(x),t(y+(j_{c}-j_{c+t^{i}(a)})-j_{t(x)})),
    \end{eqnarray*}
    for all $x,y,c\in A$. For $x'=t(x)$, $y'=t(y-j_{t(x)})$ and by (\ref{jt}), we have that
    $$f(x',y')=f(x',y'+t(j_{c}-j_{c+t^{i}(a)}))=f(x',y'+(j_{t(c)}-j_{t(c)+t^{i+1}(a)})).$$
    Hence (\ref{ti}) holds for $i+1$. By induction (\ref{ti}) holds for all $i\geq 2$.
    Since $A$ is finite, the order of $t$ is finite.  Thus
    \begin{equation*}f(x,y)=f(x,y+(j_{c}-j_{c+t^{z}(a)})),
    \end{equation*}
    all $x,y,c\in A$ and all $z\in\Z$.
    Since $V_{a,1}=\gr(j_c-j_{c+t^z(a)}\mid c\in A,\; z\in Z )$, we get that
    $$f(x,y)=f(x,y+v_1),$$
    for all $x,y\in A$ and $v_1\in V_{a,1}$. We shall prove by induction on $n$ that
    \begin{equation}
        \label{eqV} f(x,y)=f(x,y+v_n),
    \end{equation}
    for all $x,y\in A$, $v_n\in V_{a,n}$ and every positive integer $n$.
    We have proved (\ref{eqV}) for $n=1$. Suppose that (\ref{eqV}) holds for
    some positive integer $n$.
    Then we have that
    \begin{eqnarray*}f(t(u),t(v-j_{t(u)-x}))&=&f(\sigma_{(x,0)}(u,v))\\
        &=&f(\sigma_{(x,v_n)}(u,v+v_n'))\\
        &=&f(t(u)+v_n,t(v+v_n'-j_{t(u)+v_n-x})),\end{eqnarray*}
    for all $x,u,v\in A$ and $v_n,v_n'\in V_{a,n}$.
    For $u'=t(u)$ and $v'=t(v-j_{t(u)-x})$, we get that
    $$f(u',v')=f(u'+v_n,v'+t(v_n'+j_{u'-x}-j_{u'+v_n-x})).$$
     This
    implies that
    $$f(u,v)=f(u+v_n,v+t(v_n'+j_{w}-j_{v_n+w})),$$
    for all $u,v,w\in A$ and $v_n,v_n'\in V_{a,n}$. Since $t(V_{a,n})=V_{a,n}$,
    by (\ref{jt}), we have that
        $$f(x,y)=f(x+v_n,y+v_n'+j_{w}-j_{w+t(v_n)}),$$
        for all $x,y,w\in A$ and $v_n,v_n'\in V_{a,n}$.
         As $f(u,v)=f(u,v+z_n)$, we then also get
    \begin{eqnarray*}\lefteqn{f(t(x)+v,t(y-j_{t(x)+v-u}))}\\
        &=&f(\sigma_{(u,v)}(x,y))\\
        &=&f(\sigma_{(u,v+z_n)}(x+v_n,y+v_n'+j_{w}-j_{w+t(v_n)}))\\
        &=&f(t(x+v_n)+v+z_n,t(y+v_n'+j_{w}-j_{w+t(v_n)}\\
        &&\qquad-j_{t(x+v_n)+v+z_n-u})),
    \end{eqnarray*}
        for all $u,v,x,y,w\in A$ and $v_n,v_n',z_n\in V_{a,n}$. For
    $t(v_n)=-z_n$, this leads to
    $$f(t(x)+v,t(y-j_{t(x)+v-u}))=f(t(x)+v,t(y+v_n'+j_{w}-j_{w-z_n}-j_{t(x)+v-u})),$$
    for all $x,y,u,v,w\in A$ and $z_n,v_n'\in V_{a,n}$. For $x'=t(x)+v$ and $y'=t(y-j_{t(x)+v-u})$,
    we get that
    $$f(x',y')=f(x',y'+t(v_n'+j_{w}-j_{w-z_n})).$$
    Thus, by (\ref{jt}),
    $$f(x',y')=f(x',y'+t(v_n')+j_{t(w)}-j_{t(w-z_n)}),$$
    for all $x',y',w\in A$ and $z_n,v_n'\in V_{a,n}$.
    Since $t(V_{a,n})=V_{a,n}$, it follows that
    $$f(x,y)=f(x,y+v_n+j_{w}-j_{w+z_n'}),$$
     for all $x,y,w\in A$ and $v_n,z_n'\in V_{a,n}$.
    Then the equality $V_{a,n+1}=V_{a,n}+\gr(j_{w}-j_{w+z_n'}\mid w\in A,\; z_n'\in
    V_{a,n})$ implies that
    $$f(x,y)=f(x,y+v_{n+1}),$$
    for all $x,y\in A$  and $v_{n+1}\in V_{a,n+1}$, so the inductive claim (\ref{ti}) follows. By the definition of
    $V_a$ we then get that
    $$f(x,y)=f(x,y+v),$$
    for all $x,y\in A$ and $v\in V_{a}=A$. Therefore
    $$f(x,0)=f(x,y),$$
    for all $x,y\in A$. This implies that
    $$f(0,0)=f(0,-t(j_0))=f(\sigma_{(0,0)}(0,0))=f(\sigma_{(0,x)}(0,y))=f(x,t(y-j_{x}))=f(x,y),$$
    for all $x,y\in A$. Therefore $|Y|=1$ in this case, as
    desired. Hence we may assume that  $f(x,y)=f(u,v)$ and $y=v$
    imply  $x=u$.  We will show that this case leads to a
    contradiction if $|A|$ is odd or $t$ has odd order.

    Let $f(0,0)=y_0$.  Then $f^{-1}(y_0)$ is of the form $f^{-1}(y_0)=\{ (a_i,c_i)\mid i=1, \dots
    ,m\}$, where $m>1$ and $c_i\neq c_j$ for all $i\neq j$. We
    may assume that $(0,0)=(a_1,c_1)$.
    Using (\ref{sigmainverse}) and (\ref{inverse}), for every $i,u\in\{ 1,\dots, m\}$, we have that
    \begin{eqnarray}\label{y0new1}
        y_0&=&f(\sigma_{(0,0)}\sigma^{-1}_{(0,0)}(0,0))=f(\sigma_{(a_u,c_u)}\sigma^{-1}_{(a_i,c_i)}(0,0))\notag\\
        &=&f(\sigma_{(a_u,c_u)}(t^{-1}(-c_i),j_{a_i}))\notag\\
        &=&f(c_u-c_i,t(j_{a_i}-j_{c_u-c_i-a_u}))
    \end{eqnarray}
    and
        \begin{eqnarray}\label{y0new2}
        y_0&=&f(\sigma^{-1}_{(0,0)}\sigma_{(0,0)}(0,0))=f(\sigma^{-1}_{(a_i,c_i)}\sigma_{(a_u,c_u)}(0,0))\notag\\
        &=&f(\sigma^{-1}_{(a_i,c_i)}(c_u,t(-j_{c_u-a_u})))\notag\\
        &=&f(t^{-1}(c_u-c_i),j_{c_u-a_i}-j_{c_u-a_u}) .
    \end{eqnarray}
    By (\ref{y0new1}) and (\ref{y0new2}), since $|f^{-1}(y_0)|=m=|\{ c_u\mid u=1,\dots ,m\}|$,
    for every $i\in\{ 1,\dots ,m\}$ we have that
    \begin{eqnarray}\label{y0new3}
        f^{-1}(y_0)&=&\{ (a_u,c_u)\mid u=1,\dots ,m\}\notag\\
        &=&\{ (c_u-c_i,t(j_{a_i}-j_{c_u-c_i-a_u}))\mid u=1,\dots ,m\}\notag\\
        &=&\{ (t^{-1}(c_u-c_i),j_{c_u-a_i}-j_{c_u-a_u})\mid u=1,\dots
        ,m\}.
    \end{eqnarray}
    Let $S=\{ a_u\mid u=1,\dots ,m\}$. By (\ref{y0new3}), since $a_1=c_1=0$, we get that
    $$S=\{ c_u\mid u=1,\dots ,m\}=\{ t^{-1}(c_u)\mid u=1,\dots ,m\}$$
    is a subgroup of $A$ and $t(S)=S$. In particular, all $a_i, i=1,\ldots, m$,
    are different. Moreover, using (\ref{y0new1}) for $i=1$,
    \begin{eqnarray*}
        S&=&\{ c_u\mid u=1,\dots ,m\}=\{ t(j_0-j_{c_u-a_u})\mid u=1,\dots ,m\}\\
        &=&\{ j_0-j_{c_u-a_u}\mid u=1,\dots ,m\}\subseteq \{ j_0-j_a\mid a\in S\}.
    \end{eqnarray*}
    Since $|S|=m$, it follows that $|\{j_a\mid a\in S\}|=m$. Therefore,
    \begin{equation} \label{jj}
        \forall a,c\in S,\; j_a=j_c\implies a=c.
    \end{equation}
    Since $j_a=j_{-a}$, we get that $a=-a$ for all $a\in S$.
      If $|A|$ is odd, then $|S|=m=1$, a contradiction.

     Suppose that $t$ has odd order.
     For every $i\in\{ 1,\dots ,m\}$ there exists $u_i\in \{1,\dots ,m\}$ such that $c_{u_i}=t^{-1}(c_i)=t^{-1}(-c_i)$.
     Hence, using (\ref{jt}) and (\ref{sigmainverse}), we get
     \begin{eqnarray*}
     f(0,0)=y_0&=&f(\sigma_{(0,0)}\sigma^{-1}_{(a_{u_i},c_{u_i})}(t^{-1}(-c_i),j_{a_i}-j_0))\\
        &=&f(\sigma_{(0,0)}(0,t^{-1}(j_{a_i}-j_0)+j_{t^{-1}(c_i)+a_{u_i}}))\\
        &=&f(0,j_{a_i}-j_0+t(j_{t^{-1}(c_i)+a_{u_i}}-j_{0}))\\
        &=&f(0,-j_{a_i}+j_0+j_{c_i+t(a_{u_i})}-j_{0})\\
        &=&f(0,j_{c_i+t(a_{u_i})}-j_{a_i})
        \end{eqnarray*}
        where the  second last equality follows because $-j_{a_i}+j_0=j_{a_i}-j_0\in S$.
     Therefore,
       $$j_{a_i}=j_{c_i+t(a_{u_i})}$$
     and by (\ref{jj}), we get that $a_{u_i}=t^{-1}(a_i)+t^{-1}(c_i)$.
      Hence, for every $i\in\{ 1,\dots ,m\}$
     \begin{eqnarray*}
        y_0&=&f(a_i,c_i)=f(a_{u_i},c_{u_i})=f(t^{-1}(a_i)+t^{-1}(c_i),t^{-1}(c_i))\\
        &=&f(t^{-1}(t^{-1}(a_i)+t^{-1}(c_i))+ t^{-2}(c_i),t^{-2}(c_i))=f(t^{-2}(a_i),t^{-2}(c_i))\\
        &=&f(t^{-2z}(a_i),t^{-2z}(c_i)),
        \end{eqnarray*}
     for all $z\in\Z$ (the fourth and the sixth equality are obtained by repeating
     the argument).
      Since the order of $t$ is odd, there exists $z\in \Z$ such that
     $t^{-2z}=t^{-1}$. Hence
     $$f(t^{-1}(a_2)+t^{-1}(c_2),t^{-1}(c_2))=f(t^{-1}(a_2),t^{-1}(c_2)),$$
     and thus $t^{-1}(a_2)+t^{-1}(c_2)=t^{-1}(a_2)$ and then $c_2=0=c_1$, a contradiction.
         Therefore the result follows.
\end{proof}

\begin{question}
Does the converse of Theorem \ref{simple_necessary} hold for any
nonzero finite abelian group $A$?
\end{question}

\section{An alternative construction}\label{section7}

In this section we shall see an alternative construction of the
solutions constructed in Section \ref{section5}, in the case where
$t-\id$ is also an automorphism of the group $A$. This
construction uses the asymmetric product of left braces, in
particular, this answers in affirmative Question 7.4 stated in
\cite{CO21} in the case where $t-1$ is invertible in $\Z/(n)$.
Furthermore, this allows us to describe the left brace structure
of the permutation groups of these solutions.

Let $A$ be a nontrivial additive abelian group. Let
$$H=\{ f\colon A\longrightarrow \Z\mid f(x)\neq 0 \mbox{ for finitely many }x\in A\}.$$
We define a sum on $H$ by $(f+g)(x)=f(x)+g(x)$, for all $f,g\in H$
and all $x\in A$. Let $t\in\Aut(A)$  be such that
$t-\id\in\Aut(A)$. Let $T$ be the subgroup of $\Aut(A)$ generated
by $t$.

Let $A_1=A\rtimes T$ be the semidirect product of the trivial left
braces on the groups $A$ and $T$. Let $(j_a)_{a\in A}$ be a family
of elements  of $A$ such that $j_a=j_{-a}$  and
\begin{equation} \label{jbis}
    j_{t^z(a)}=t^z(j_a)-(t^z-id)(j_0),
\end{equation}
for all $a\in A$ and all $z\in \Z$.  For every $a\in A$ and
$z\in\Z$, let $g_{(a,t^z)}\colon A\longrightarrow A$ be the map
defined by
$$g_{(a,t^z)}(x)=t^{-z}\left(x-a-(t-\id)^{-2}(t^{z+1}-t)(j_0)\right),$$
for all $x\in A$. Let $\alpha\colon A_1\longrightarrow \Aut(H,+)$
be the map defined by
$$\alpha(a,t^z)(f)=f\circ g_{(a,t^z)},$$
for all $a\in A$, $f\in H$ and $z\in\Z$.
 By the definition of the
sum in $H$, it is clear that $\alpha(a,t^z)$ is an automorphism of
$H$ because $g_{(a,t^z)}$ is bijective. Let $a_1,a_2\in A$, $f\in
H$ and let $z_1,z_2\in \Z$. Note that
\begin{eqnarray*}\lefteqn{g_{(a_2,t^{z_2})}g_{(a_1,t^{z_1})(x)}}\\
    &=&g_{(a_2,t^{z_2})}\left(t^{-z_1}\left(x-a_1-(t-\id)^{-2}(t^{z_1+1}-t)(j_0)\right)\right)\\
    &=&t^{-z_2}\left( t^{-z_1}\left(x-a_1-(t-\id)^{-2}(t^{z_1+1}-t)(j_0)\right)-a_2-(t-\id)^{-2}(t^{z_2+1}-t)(j_0)\right)\\
    &=&t^{-z_1-z_2}\left(x-a_1-(t-\id)^{-2}(t^{z_1+1}-t)(j_0)-t^{z_1}(a_2)-t^{z_1}(t-\id)^{-2}(t^{z_2+1}-t)(j_0)\right)\\
    &=&t^{-(z_1+z_2)}\left(x-(a_1+t^{z_1}(a_2))-(t-\id)^{-2}(t^{z_1+z_2+1}-t)(j_0)\right)\\
    &=&g_{(a_1+t^{z_1}(a_2),t^{z_1+z_2})}(x)\\
    &=&g_{(a_1,t^{z_1})(a_2,t^{z_2})}(x),
\end{eqnarray*}
for all $x\in A$. Hence $$g_{(a_2,t^{z_2})}\circ
g_{(a_1,t^{z_1})}=g_{(a_1,t^{z_1})(a_2,t^{z_2})}$$
    and thus
    \begin{eqnarray*}
        \alpha(a_1,t^{z_1})\alpha(a_2,t^{z_2})(f)&=&\alpha(a_1,t^{z_1})(f\circ g_{(a_2,t^{z_2})})\\
        &=&f\circ g_{(a_2,t^{z_2})}\circ g_{(a_1,t^{z_1})}\\
        &=&f\circ g_{(a_1,t^{z_1})(a_2,t^{z_2})}\\
        &=&\alpha((a_1,t^{z_1})(a_2,t^{z_2}))(f),
    \end{eqnarray*}
    for all $a_1,a_2\in A$, $z_1,z_2\in\Z$ and $f\in H$.
 Hence $\alpha$ is a group homomorphism from the multiplicative
group of $A_1$ to $\Aut(H,+)$.

Consider the semidirect product
$$H \rtimes_{\alpha} A_1.$$
Consider the bilinear form
$$b_{(j_a)_{a\in A}}\colon H \times H \longrightarrow A_1,$$
defined by
$$b_{ (j_a)_{a\in a}}(f,h)=\left(\sum_{x,y\in A}f(x)h(y)t(j_{x-y}-j_0),\id\right),$$
for all $f,h\in H$. Note that, $g_{(c,t^z)}$ is bijective and
    \begin{equation}\label{g-1}g^{-1}_{(c,t^z)}(x)=t^z(x)+c+(t-\id)^{-2}(t^{z+1}-t)(j_0),
        \end{equation}
    for all $c,x\in A$ and $z\in\Z$. In the left brace $A_1$, for $c\in
A$ and $z\in \Z$, we have
\begin{align*}\lambda_{(c,t^z)}(b_{(j_a)_{a\in A}}(f,h))=&(c,t^z)\left(\sum_{x,y\in A}f(x)h(y)t(j_{x-y}-j_0),\id\right)-(c,t^z)\\
    =&\left(t^z\left(\sum_{x,y\in A}f(x)h(y)t(j_{x-y}-j_0)\right),\id\right)\\
    =&\left(\sum_{x,y\in A}f(x)h(y)t^{z+1}(j_{x-y}-j_0),\id\right)
    \end{align*}
and
\begin{eqnarray}\label{newbilin}\lefteqn{b_{(j_a)_{a\in A}}(\alpha(c,t^z)(f),\alpha(c,t^z)(h))}\notag\\
    &=&\left(\sum_{x,y\in A}\alpha(c,t^z)(f)(x)\alpha(c,t^z)(h)(y)t(j_{x-y}-j_0),\id\right)\notag\\
    &=&\left(\sum_{x,y\in A}f\left(g_{(c,t^z)}(x)\right)h\left(g_{(c,t^z)}(y)\right)t(j_{x-y}-j_0),\id\right)\notag\\
    &=&\left(\sum_{x,y\in A}f(x)h(y)t(j_{g^{-1}_{(c,t^z)}(x)-g^{-1}_{(c,t^z)}(y)}-j_0),\id\right)\notag\\
    &=&\left(\sum_{x,y\in A}f(x)h(y)t(j_{t^z(x)-t^z(y)}-j_0),\id\right)\notag\\
    &=&\left(\sum_{x,y\in A}f(x)h(y)t^{z+1}(j_{x-y}-j_0),\id\right),
\end{eqnarray}
where  the third equality follows because $g_{(c,t^z)}$ is
bijective, so that we may replace $g_{(c,t^z)}(x)$  by $x$ and
$g_{(c,t^z)}(y)$  by $y$, the fourth equality follows from
(\ref{g-1}) and the last equality follows from (\ref{jbis}). Hence
\begin{equation}\label{asymmetric}\lambda_{(c,t^z)}(b_{(j_a)_{a\in A}}(f,h))=b_{(j_a)_{a\in A}}(\alpha(c,t^z)(f),\alpha(c,t^z)(h)),
\end{equation}
for all $f,h\in H$ and all $(c,t^z)\in A_1$, and thus we can
consider the asymmetric product  (\cite{CCS}) of the trivial brace
$H$ and the left brace $A_1$, defined via $\alpha$ and
$b_{(j_a)_{a\in A}}$:
$$B_{(j_a)_{a\in A}} = H \rtimes_{\circ} A_1.$$
Thus, $(B_{(j_a)_{a\in A}},\circ) \cong H \rtimes_{\alpha} A_1$
and the addition is defined by
$$(f,(c_1,t^{z_1}))+(h,(c_2,t^{z_2})) =(f+h, (c_1+c_2,t^{z_1+z_2})+b_{(j_a)_{a\in A}}(f,h)),$$
for all $f,h \in  H$, $c_1,c_2 \in A$ and integers $z_1,z_2$. We
know that $(B_{(j_a)_{a\in A}},+,\circ)$ is a left brace. We
denote the lambda map of $B_{(j_a)_{a\in A}}$ by
$\lambda^{(j_a)_{a\in A}}$. Then we have
\begin{eqnarray}\lefteqn{\lambda^{(j_a)_{a\in A}}_{(f,(c_1,t^{z_1}))}(h,(c_2,t^{z_2}))}\nonumber\\
    &=& (f,(c_1,t^{z_1})) (h,(c_2,t^{z_2})) -(f,(c_1,t^{z_1})) \nonumber \\
    &=&(f+\alpha(c_1,t^{z_1})(h), (c_1+t^{z_1}(c_2),t^{z_1+z_2}))-(f,(c_1,t^{z_1})) \nonumber \\
    &=&(\alpha (c_1,t^{z_1})(h),(t^{z_1}(c_2),t^{z_2})-b_{(j_a)_{a\in A}}(f,\alpha (c_1,t^{z_1})(h))) . \label{lambdauibis}
\end{eqnarray}
For every $a\in A$, we define $e_a\colon A\longrightarrow \Z$ by
$$e_a(x)=\delta_{a,x}=\left\{ \begin{array}{lr}
    1&\quad\mbox{if }x=a\\
    0&\quad\mbox{otherwise}\end{array}\right.$$
    Thus $e_a\in H$ for all $a\in A$. Note that
    $$e_a(g_{(c,t^z)}(x))=\delta_{a,g_{(c,t^z)}(x)}=
    \delta_{g^{-1}_{(c,t^z)}(a),x}=e_{g^{-1}_{(c,t^z)}(a)}(x),$$
    for all $a,c,x\in A$ and $z\in\Z$. Hence, by the definition of
     $\alpha_{(c,t^z)}$,
    $$\alpha_{(c,t^z)}(e_a)=e_{g^{-1}_{(c,t^z)}(a)},$$
    for all $a,c\in A$ and $z\in \Z$.

Put
$$x_{(a,c)} =(e_a,(c-t(t-\id)^{-1}(j_0),t))\in B_{(j_a)_{a\in A}}.$$
Using (\ref{lambdauibis}) we get
\begin{eqnarray}
    \lefteqn{\lambda^{(j_a)_{a\in A}}_{x_{(a_1,a_2)}} (x_{(c_1,c_2)})}\\
     &=& \lambda^{(j_a)_{a\in A}}_{(e_{a_1},(a_2-t(t-\id)^{-1}(j_0),t))} (e_{c_1},(c_2-t(t-\id)^{-1}(j_0),t))\nonumber\\
    &=& (\alpha(a_2-t(t-\id)^{-1}(j_0),t)(e_{c_1}),\nonumber\\
    &&(t(c_2-t(t-\id)^{-1}(j_0)),t)-b_{(j_a)_{a\in A}}(e_{a_1},\alpha(a_2-t(t-\id)^{-1}(j_0),t)(e_{c_1})))\nonumber\\
    &=&(e_{t(c_1)+a_2},(t(c_2-t(t-\id)^{-1}(j_0)),t)-b_{(j_a)_{a\in A}}(e_{a_1},e_{t(c_1)+a_2}))\nonumber \\
    &=&(e_{t(c_1)+a_2},(t(c_2)-t^2(t-\id)^{-1}(j_0)-t(j_{a_1-t(c_1)-a_2}-j_0),t))\nonumber \\
    &=&(e_{t(c_1)+a_2},t(c_2-j_{a_1-t(c_1)-a_2})-t(t-\id)^{-1}(j_0),t)\nonumber\\
    &=& x_{(t(c_1)+a_2, t(c_2-j_{t(c_1)+a_2-a_1}))}. \label{action1bis}
\end{eqnarray}
Let $X=\{ x_{(a,c)}\mid a,c\in A\}$ and
$$r_{(j_a)_{a\in A}}:X\times X \rightarrow X\times X$$
be the map defined by
$$r_{(j_a)_{a\in A}}(x_{(a_1,a_2)}, x_{(c_1,c_2)}) =
(\lambda^{(j_a)_{a\in A}}_{x_{(a_1,a_2)}} (x_{(c_1,c_2)}),
(\lambda^{(j_a)_{a\in A}}_{\lambda^{(j_a)_{a\in
A}}_{x_{(a_1,a_2)}}(x_{(c_1,c_2)})})^{-1}(x_{(a_1,a_2} )). $$
Note that
$r_{(j_a)_{a\in A}}$ is the restriction to $X^{2}$ of the solution
associated to the left brace $B_{(j_a)_{a\in A}}$ (see \cite[Lemma
2]{CJOComm}). Thus $(X,r_{(j_a)_{a\in A}})$ is a solution of the
YBE.

\begin{proposition}
    The solution $(X,r_{(j_a)_{a\in A}})$ of the YBE is isomorphic to the solution $(A^2,r)$ defined in Section \ref{section5}.
\end{proposition}

\begin{proof}
    Recall that
    $$r((a_1,a_2),(c_1,c_2))=(\sigma_{(a_1,a_2)}(c_1,c_2),\sigma^{-1}_{\sigma_{(a_1,a_2)}(c_1,c_2)}(a_1,a_2)),$$
    where
    $\sigma_{(a_1,a_2)}(c_1,c_2)=(t(c_1)+a_2,t(c_2-j_{t(c_1)+a_2-a_1}))$.
    Consider the map $f\colon A^2\longrightarrow X$ defined by
$f(a,c)=x_{(a,c)}$, for all $a,c\in A$. Then by
(\ref{action1bis}), $f$ is an isomorphism from
the solution $(A^2,r)$, defined in Section \ref{section5}, to
$(X,r_{(j_a)_{a\in A}})$.
\end{proof}

 In particular this result answers in affirmative Question 7.4
stated in \cite{CO21}.

Now we shall study the structure of left brace of
$\mathcal{G}(X,r_{(j_a)_{a\in A}})$. We shall prove that if $A$
has finite exponent $n$ and $t$ has finite order coprime with $n$,
then $\mathcal{G}(X,r_{(j_a)_{a\in A}})$ is an asymmetric product
of a trivial brace $\overline{H}$ and the semidirect product
$A_1=A\rtimes T$ of the trivial braces on the groups $A$ and $T$.

 Note that if $A$ has finite exponent $n$, then we can replace
$H$ in the above construction by
$$H'=\{ f\colon A\longrightarrow \Z/(n)\mid f(x)\neq 0 \mbox{ for finitely many }x\in A\}.$$
Indeed, we define a sum on $H'$ by $(f+h)(x)=f(x)+h(x)$, for all $f,h\in H'$ and $x\in A$,
and $\alpha'\colon A_1\longrightarrow \Aut(H',+)$ by $\alpha'(a,t^z)(f)=f\circ g_{(a,t^z)}$
for all $a\in A$, $z\in\Z$ and $f\in H'$. Then $\alpha'$ is a homomorphism from the multiplicative
group of $A_1$ to $\Aut(H',+)$. We define
$$b'_{(j_a)_{a\in A}}\colon H' \times H' \longrightarrow A_1,$$
by
$$b'_{ (j_a)_{a\in a}}(f,h)=\left(\sum_{x,y\in A}f(x)h(y)t(j_{x-y}-j_0),\id\right),$$
for all $f,h\in H'$. Since $A$ has exponent $n$ and
$f(x)h(y)\in\Z/(n)$, $b'_{(j_a)_{a\in A}}$ is a well-defined
bilinear form. Then (\ref{asymmetric}) is satisfied for $\alpha'$
and $b'_{(j_a)_{a\in A}}$ and thus we can consider the asymmetric
product $B'_{(j_a)_{a\in A}}=H'\rtimes_{\circ} A_1$ via $\alpha'$
and $b'_{(j_a)_{a\in A}}$. For every $a\in A$, we define
$e'_a\colon A\longrightarrow \Z/(n)$ by
$$e'_a(x)=\delta_{a,x}=\left\{ \begin{array}{lr}
    1&\quad\mbox{if }x=a\\
    0&\quad\mbox{otherwise}\end{array}\right.$$
Thus $e_a\in H'$ for all $a\in A$. Put
$$x'_{(a,c)} =(e'_a,(c-t(t-\id)^{-1}(j_0),t))\in B'_{(j_a)_{a\in A}}.$$
Let $X'=\{ x'_{(a,c)}\mid a,c\in A\}$ and
$$r'_{(j_a)_{a\in A}}:X'\times X' \rightarrow X'\times X'$$
be the restriction to $(X')^2$ of the solution  associated to the left brace $B'_{(j_a)_{a\in A}}$.
One can see that the map $X\longrightarrow X' ,\; x_{(a,c)}\mapsto x'_{(a,c)}$
is an isomorphism of solutions from $(X,r_{(j_a)_{a\in A}})$ to $(X',r'_{(j_a)_{a\in A}})$. Hence
$$\mathcal{G}(X,r_{(j_a)_{a\in A}})\cong\mathcal{G}(X',r'_{(j_a)_{a\in A}})$$
as left braces.

The advantage of dealing with $X', r'$ and $H'$ in place of $X, r,
H$ is that in the most interesting case when $X$ is a finite set
the group $H'$ is also finite. Moreover, in an important special
case $\mathcal{G}(X,r_{(j_a)_{a\in A}})$ comes from an asymmetric
product of trivial left braces, as we shall see in
Theorem~\ref{newasym}, and it is often isomorphic to
$B'_{(j_a)_{a\in A}}$, see Proposition~\ref{Bcbis}.

 Note that $\gr(X')_+$ is a left ideal of the left
brace $B'_{(j_a)_{a\in A}}$, in particular $\gr(X')_+$ is a left
brace with the operations inherited from $B'_{(j_a)_{a\in A}}$. It
is difficult to study the structure of this left subbrace of
$B'_{(j_a)_{a\in A}}$ in general. However, in the next
lemma we show that in some important cases one can solve this
problem because $B'_{(j_a)_{a\in A}}=\gr(X')_+$ and the structure
of left brace of $B'_{(j_a)_{a\in A}}$ is known.  This result
is needed in the proof of Proposition \ref{Bcbis}, which is the
key to prove Theorem \ref{newasym}.

\begin{lemma}
    \label{generatorsbis} Suppose that $A$ has finite exponent $n>1$.
    If $t$ has finite order coprime with $n$, then $B'_{(j_a)_{a\in A}}=\gr(X')_+$.
\end{lemma}

\begin{proof} Let $G=\gr(X')_+$.
    We have that
    \begin{align*}x'_{(a,c)}-x'_{(a,0)}&=(e'_a,(c-t(t-\id)^{-1}(j_0),t))-(e'_a,(-t(t-\id)^{-1}(j_0),t))\\
        &=(0,(c-t(t-\id)^{-1}(j_0),t)-(-t(t-\id)^{-1}(j_0),t)-b'_{(j_a)_{a\in A}}(e'_a,0))\\
        &=(0,(c,\id)),
    \end{align*}
    for all $a,c\in A$.  Hence
    $\{ (0,(c,\id))\mid c\in A\}\subseteq G$. Note that
    \begin{align*}x_{(a,t(t-\id)^{-1}(j_0))}+x_{(a,t(t-\id)^{-1}(j_0))}&=(e'_a,(0,t))+(e'_a,(0,t))\\
        &=(2e'_a,2(0,t)+b'_{(j_a)_{a\in A}}(e'_a,e'_a))\\
        &=(2e'_a,(0,t^2)) \in G.
    \end{align*}
    By induction on $i$ one can prove that
    $$ix'_{(a,t(t-\id)^{-1}(j_0))}=(ie'_a,(0,t^{i}))\in G,$$
    for all positive integer $i$. Since $nf=0$ for all $f\in H'$, we have that
    $$nx'_{(a,t(t-\id)^{-1}(j_0))}=(ne'_a,(0,t^n))= (0,(0,t^n))\in
    G.$$
    Since $n$ is coprime with the order of $t$, there exists an integer $z$ such that
    $t^{nz}=t$. Hence
    $$(0,(0,t^n))^{zs}=(0,(0,t^{nzs}))=(0,(0,t^s)),$$
    for all $s\in\Z$. Thus we have  that
    $$(0,(a,t^s))=(0,(a,\id))+(0,(0,t^s))\in G,$$
    for all $a\in A$ and all $s\in \Z$.
    Let $f\in H'\setminus\{ 0\}$. There exist a positive integer $k$,  distinct elements
    $a_1,\dots ,a_k\in A$
     and $z_1,\dots ,z_k\in\Z\setminus\{ 0\}$ such that $f=\sum_{i=1}^kz_ie_{a_i}$.  Hence
    \begin{align*}(f,(a,t^s))&=\left(\sum_{i=1}^kz_ie_{a_i},(a,t^s)\right)\\
        &=(e_{a_1},(0,\id))^{z_1}\cdots (e_{a_k},(0,\id))^{z_k}(0,(a,t^s))\in G,
        \end{align*}
    for all $f\in H'$, $s\in \Z$ and all $a\in A$, because
    $$(e_c,(0,\id))=(e_c,(0,t))+(0,(0,t^{-1}))=x_{(c,t(t-\id)^{-1}(j_0))}+(0,(0,t^{-1}))\in G,$$
    for all $c\in A$,
    and $G$ is a left ideal of $B'_{(j_a)_{a\in A}}$.
    Therefore the result follows.
\end{proof}

Our next result is an easy consequence of Propositions
\ref{newindecomposable} and \ref{newirretractable}.

\begin{proposition}\label{indbis}
    If for every nonzero $a\in A$ there exists $c\in A$ such that $j_{a+c}-j_c\neq 0$ and
    $A=\gr(j_a-j_c\mid a,c\in A)_+$, then the solution
    $(X',r'_{(j_a)_{a\in A}})$ is indecomposable and irretractable.
\end{proposition}

\begin{proposition}\label{Bcbis}
    Suppose that $A$ has finite exponent $n>1$.
    Assume that the order of $t$ and $n$ are coprime. Then
    $$B'_{(j_a)_{a\in A}}/\soc(B'_{(j_a)_{a\in A}})\cong\mathcal{G}(X',r'_{(j_a)_{a\in A}})$$
    as left braces and
    $$\soc(B'_{(j_a)_{a\in A}})=
    \{ (f,(0,\id))\in B'_{(j_a)_{a\in A}}\mid b'_{(j_a)_{a\in A}}(f,h)=(0,\id) \mbox{ for all  }h\in H'\}.$$
    In particular, the left braces $B'_{(j_a)_{a\in A}}$ and
    $\mathcal{G}(X',r'_{(j_a)_{a\in A}})$ are isomorphic if and only if $b'_{(j_a)_{a\in A}}$ is non-singular.
    \end{proposition}
\begin{proof}
    Note that by the definition of the lambda map of $B'_{(j_a)_{a\in A}}$  and the definition of
    $\alpha'$, it is easy to see that $(f,(a,t^s))\in \soc(B'_{(j_a)_{a\in A}})$
    if and only if $(a,t^s)=(0,\id)$ and $b'_{(j_a)_{a\in A}}(f,h)=(0,\id)$,
    for all $h\in H'$. Thus
    $$\soc(B'_{(j_a)_{a\in A}})=
    \{ (f,(0,\id))\in B'_{(j_a)_{a\in A}}\mid b'_{(j_a)_{a\in A}}(f,h)=(0,\id)
    \mbox{ for all  }h\in H'\}.$$

    Let
    $(B'_{(j_a)_{a\in A}},r_{B'_{(j_a)_{a\in A}}})$ be the solution
    associated to the left brace $B'_{(j_a)_{a\in A}}$, that is,
    $r_{B'_{(j_a)_{a\in A}}}\colon B'_{(j_a)_{a\in A}}\times
    B'_{(j_a)_{a\in A}}\longrightarrow B'_{(j_a)_{a\in A}}\times
    B'_{(j_a)_{a\in A}}$ is defined by
    \begin{eqnarray*}\lefteqn{r_{B'_{(j_a)_{a\in A}}}((f,(a_1,t^{s_1})),(g,(a_2,t^{s_2})))}\\
        &=& (\lambda_{(f,(a_1,t^{s_1}))} (g,(a_2,t^{s_2})),
        \lambda^{-1}_{\lambda_{(f,(a_1,t^{s_1}))}(g,(a_2,t^{s_2}))} (f,(a_1,t^{s_1}))),\end{eqnarray*}
    for all $f,g\in H'$, $a_1,a_2\in A$
    and integers $s_1,s_2$, where $\lambda$ is the lambda map of the left brace $B'_{(j_a)_{a\in A}}$.
    It is well-known that
    $$B'_{(j_a)_{a\in A}}/\soc(B'_{(j_a)_{a\in A}})\cong\mathcal{G}(B'_{(j_a)_{a\in A}},r_{B'_{(j_a)_{a\in A}}})$$
    as left braces. Since the order of $t$ and $n$ are coprime, by Lemma \ref{generatorsbis},
    the map
    $$\varphi\colon \mathcal{G}(B'_{(j_a)_{a\in A}},r_{B'_{(j_a)_{a\in A}}})\longrightarrow \mathcal{G}(X',r_{(j_a)_{a\in A}})$$
    defined by
    $\varphi(\lambda_{(f,(a,t^s))})= \lambda_{(f,(a,t^s))}|_{X'}$,
    for all $(f,(a,t^s))\in B'_{(j_a)_{a\in A}}$, is an isomorphism of
    multiplicative groups.

    Furthermore, the inclusion map
    $X'\longrightarrow B'_{(j_a)_{a\in A}}$ is a homomorphism of
    solutions from $(X',r'_{(j_a)_{a\in A}})$ to
    $(B'_{(j_a)_{a\in A}},r_{B'_{(j_a)_{a\in A}}})$.
    Hence there exists a homomorphism of left braces
    $\varphi'\colon G(X',r'_{(j_a)_{a\in A}})\longrightarrow B'_{(j_a)_{a\in A}}$ such that
    $\varphi'(x)=x$, for all $x\in X'$.
    By Lemma \ref{generatorsbis}, $X'$ generates the additive group of
    $B'_{(j_a)_{a\in A}}$.
    Therefore $\varphi'$ is surjective and thus
    $\varphi'(\soc(G(X',r'_{(j_a)_{a\in A}})))\subseteq \soc(B'_{(j_a)_{a\in A}})$. Hence it
    induces a homomorphism of left braces
    $$\mathcal{G}(X',r'_{(j_a)_{a\in A}})\longrightarrow
    \mathcal{G}(B'_{(j_a)_{a\in A}},r_{B'_{(j_a)_{a\in A}}}),$$
    which is exactly $\varphi^{-1}$. Therefore the result follows.
\end{proof}

Now we can prove the main result of this section.

\begin{theorem}\label{newasym}
    Suppose that $A$ has finite exponent $n>1$.
    Assume that the order of $t$ and $n$ are coprime. Let
    $$H_1=\{ f\in H'\mid b'_{(j_a)_{a\in A}}(f,h)=(0,\id) \mbox{ for all  }h\in H'\}$$
    Then $H_1$ is a subgroup of $H'$. Let $\overline{\alpha}\colon A_1\longrightarrow \Aut(H'/H_1)$
    be the map defined by
    $$\overline{\alpha}(a,t^s)(f+H_1)=\alpha'(a,t^s)(f)+H_1,$$
    for all $f\in H'$, $a\in A$ and any integer $s$.
    Let $\overline{b_{(j_a)_{a\in A}}}\colon H'/H_1\times H'/H_1\longrightarrow A_1$
    be the map defined by
    $$\overline{b_{(j_a)_{a\in A}}}(f+H_1,g+H_1)=b'_{(j_a)_{a\in A}}(f,g),$$
    for all $f,g\in H'$. Then
    $$\mathcal{G}(X',r'_{(j_a)_{a\in A}})\cong H'/H_1\rtimes_{\circ} A_1,$$
    the asymmetric product of the trivial brace $H'/H_1$ and the left brace $A_1$
    via $\overline{\alpha}$ and $\overline{b_{(j_a)_{a\in A}}}$, as left braces.
\end{theorem}

\begin{proof}
    Let $f\in H_1$ and $h\in H'$. Note that  (\ref{newbilin}) is also satisfied by $\alpha'$ and $b'_{(j_a)_{a\in A}}$,
    thus we have that
    \begin{align*}
        b'_{(j_a)_{a\in A}}(\alpha'(c,t^s)(f),\alpha'(c,t^s)(h))&=\left(\sum_{x,y\in A}f(x)h(y)t^{s+1}(j_{x-y}-j_0),\id\right)\\
        &=\left(t^s\left(\sum_{x,y\in A}f(x)h(y)t(j_{x-y}-j_0)\right),\id\right)\\
        &=(0,t^s)(b'_{(j_a)_{a\in A}}(f,h))(0,t^{-s})\\
        &=(0,t^s)(0,\id)(0,t^{-s})=(0,\id),
    \end{align*}
    for all $c\in A$ and any integer $s$.
    Hence $\alpha'(c,t^s)(H_1)\subseteq H_1$, for all $c\in A$ and $s\in \Z$. Thus
    $$H_1=\alpha'(c,t^s)(\alpha'((c,t^s)^{-1})(H_1)\subseteq \alpha'(c,t^s)(H_1)$$
    and this proves that $\alpha'(c,t^s)(H_1)= H_1$, for all $c\in A$ and $s\in \Z$.
    Therefore $\overline{\alpha}$ is well-defined and
    it is a group homomorphism from the multiplicative group of the left brace $A_1$
    to $\Aut(H'/H_1)$. Clearly $\overline{b_{(j_a)_{a\in A}}}$ also is a well-defined bilinear form.
    By (\ref{asymmetric}) and the definitions of $\overline{\alpha}$ and
    $\overline{b_{(j_a)_{a\in A}}}$, we have that
    $$\lambda_{(c,t^s)}(\overline{b_{(j_a)_{a\in A}}}(f+H_1,g+H_1))=
    \overline{b}_{(j_a)_{a\in A}}(\overline{\alpha}(c,t^s)(f+H_1),\overline{\alpha}(c,t^s)(g+H_1)),$$
    for all $f,g\in H'$ and all $(c,t^s)\in A_1$, and thus we
    can consider the asymmetric product  (\cite{CCS}) of the trivial
    brace $H'/H_1$ and the left brace $A_1$ defined via
    $\overline{\alpha}$ and $\overline{b_{(j_a)_{a\in A}}}$:
    $$\overline{B_{(j_a)_{a\in A}}} = H'/H_1 \rtimes_{\circ} A_1.$$
    Let $\psi\colon B'_{(j_a)_{a\in A}}\longrightarrow  \overline{B_{(j_a)_{a\in A}}}$
    be the map defined by $\psi(f,(c,t^s))=(f+H_1,(c,t^s))$, for all $f\in H'$, $c\in A$
    and any integer $s$. It is easy to see that $\psi$ is a homomorphism of the multiplicative
    groups of these left braces.
    Also we have that
    \begin{eqnarray*}\lefteqn{
        \psi((f,(c_1,t^{s_1}))+(g,(c_2,t^{s_2})))}\\
        &=&\psi(f+g,(c_1+c_2,t^{s_1+s_2})+b'_{(j_a)_{a\in A}}(f,g))\\
        &=&((f+g)+H_1,(c_1+c_2,t^{s_1+s_2})+\overline{b_{(j_a)_{a\in A}}}(f+H_1,g+H_1))\\
        &=&(f+H_1,(c_1,t^{s_1}))+(g+H_1,(c_2,t^{s_2}))\\
        &=&\psi(f,(c_1,t^{s_1}))+\psi(g,(c_2,t^{s_2})),
    \end{eqnarray*}
    for all $f,g\in H'$, $c_1,c_2\in A$ and integers $s_1,s_2$. Hence $\psi$ is a
    homomorphism of left braces.
    Note that
    $$\ker(\psi)=\{ (f,(0,\id))\mid f\in H_1\}=\soc(B'_{(j_a)_{a\in A}})$$
    Hence
    $$\overline{B_{(j_a)_{a\in A}}}\cong B'_{(j_a)_{a\in A}}/\soc(B'_{(j_a)_{a\in A}}),$$
    and thus the result follows from Proposition \ref{Bcbis}.
\end{proof}

\begin{remark}
    {\rm In Section 4 of \cite{CO22} there is an alternative construction of the solutions
    $(A^2,r)$ described in Section 3 of \cite{CO22} using asymmetric products of the form
    $H\rtimes_{\circ} A$, where $A$ is a commutative ring and
    $$H=\{ f\colon A\longrightarrow A\mid f(x)\neq 0 \mbox{ for finitely many }x\in A\}.$$
    If we replace $A$ by any additive nontrivial abelian group $A'$ of finite exponent $n$ and this $H$ by
      $$H'=\{ f\colon A'\longrightarrow \Z/(n)\mid f(x)\neq 0 \mbox{ for finitely many }x\in A'\},$$
      then we can obtain similar results as Proposition \ref{Bcbis}  and
      Theorem~\ref{newasym}.

      In particular, this implies that the permutation groups of the solutions
      constructed in \cite{CO22}, for abelian groups of finite exponent, are asymmetric products of two trivial
      braces.}
      \end{remark}

\section{Examples of finite simple solutions of non-square cardinality}\label{section8}

  It has not been known whether there
exists a simple solution of the YBE of cardinality $m^2n$ for
integers $n,m>1$. This was claimed in Theorem 4.12 in \cite{CO21},
but the proof was incorrect. On the other hand, in \cite[Theorem
1.1]{CO23} it is proved that every finite indecomposable solution
of the YBE of square-free cardinality is a multipermutation
solution. In view of Lemma \ref{simple indec}, the next result can
be considered as the first step to construct infinite families of
simple solutions of cardinality $m^2n$ for integers $n,m>1$. This
aim will be then accomplished in Theorem~\ref{new-correction}.

\begin{proposition}\label{newsolnew}
    Let $m,n$ be integers such that $n,m>1$ and assume that $\Aut(\Z/(n))$ has an element $t$
    of order $m$. Let $X=\Z/(n)\times\Z/(m)\times\Z/(m)$ and let
    $\sigma_{(i,a,\mu)}\colon X\longrightarrow X$ be the maps defined by
    $$\sigma_{(i,a,\mu)}(j,c,\nu)=(t^{\mu}(j)+t^{a}(1), c+\mu, \nu+1-\delta_{i,t^{\mu}(j)+t^{a}(1)}),$$
    for all $i,j\in\Z/(n)$ and all $a,c,\mu,\nu\in \Z/(m)$.
    Let $r\colon X\times X\longrightarrow X\times X$ be the map defined by
    $$r((i,a,\mu),(j,c,\nu))=(\sigma_{(i,a,\mu)}(j,c,\nu),\sigma^{-1}_{\sigma_{(i,a,\mu)}(j,c,\nu)}(i,a,\mu)),$$
    for all $i,j\in\Z/(n)$ and all $a,c,\mu,\nu\in\Z/(m)$. Then $(X,r)$ is an indecomposable and irretractable solution of the YBE.
\end{proposition}

\begin{proof}
    First we shall check that $(X,r)$ is a solution of the YBE. Note that $\sigma_{(i,a,\mu)}\in\Sym_X$, in fact
    $$\sigma^{-1}_{(i,a,\mu)}(j,c,\nu)=(t^{-\mu}(j-t^{a}(1)),c-\mu, \nu -1+\delta_{i,j}),$$
    for all $i,j\in\Z/(n)$ and all $a,c,\mu,\nu\in\Z/(m)$.
    Let $i,j,x\in \Z/(n)$ and $a,c,y,\mu,\nu,z\in\Z/(m)$. We have that
    \begin{eqnarray*}
        \lefteqn{\sigma_{(i,a,\mu)}\sigma_{\sigma^{-1}_{(i,a,\mu)}(j,c,\nu)}(x,y,z)}\\
        &=&\sigma_{(i,a,\mu)}\sigma_{(t^{-\mu}(j-t^{a}(1)),c-\mu, \nu -1+\delta_{i,j})}(x,y,z)\\
        &=&\sigma_{(i,a,\mu)}(t^{\nu -1+\delta_{i,j}}(x)+t^{c-\mu}(1),y+\nu -1+\delta_{i,j},\\
        &&\qquad z+1-\delta_{t^{-\mu}(j-t^{a}(1)),t^{\nu -1+\delta_{i,j}}(x)+t^{c-\mu}(1)})\\
        &=&\sigma_{(i,a,\mu)}(t^{\nu -1+\delta_{i,j}}(x)+t^{c-\mu}(1),y+\nu -1+\delta_{i,j},\\
        &&\qquad z+1-\delta_{j,t^{\nu+\mu -1+\delta_{i,j}}(x)+t^{c}(1)+t^{a}(1)})\\
        &=&(t^{\nu+\mu -1+\delta_{i,j}}(x)+t^{c}(1)+t^{a}(1),y+\nu+\mu -1+\delta_{i,j},\\
        &&\qquad z+1-\delta_{j,t^{\nu+\mu -1+\delta_{i,j}}(x)+t^{c}(1)+t^{a}(1)}+1-\delta_{i,t^{\nu+\mu -1+\delta_{i,j}}(x)+t^{c}(1)+t^{a}(1)}).
    \end{eqnarray*}
    This and a symmetric argument imply that
    \begin{eqnarray} \label{YB}
        \sigma_{(i,a,\mu)}\sigma_{\sigma^{-1}_{(i,a,\mu)}(j,c,\nu)}=
        \sigma_{(j,c,\nu)}\sigma_{\sigma^{-1}_{(j,c,\nu)}(i,a,\mu)},
    \end{eqnarray}
    for all $i,j\in\Z/(n)$ and all $a,c,\mu,\nu\in\Z/(m)$.
    Since $X$ is finite, (\ref{YB}) implies that $(X,r)$ is a solution of the YBE.

    We have that
    $$\sigma_{(k,0,0)}\cdots\sigma_{(1,0,0)}(0,0,0)=\sigma_{(k,0,0)}\cdots \sigma_{(2,0,0)}(1,0,0)=\dots =(k,0,0),$$
    $$\sigma_{(k,c,c)}(t^{-c}(k)-1,0,0)=(k,c,0)$$
    and
    $$\sigma_{(k+l,0,0)}\dots\sigma_{(k,0,0)}(k,c,0)=\sigma_{(k+l,0,0)}\dots \sigma_{(k+1,0,0)}(k+1,c,1)=(k+l+1,c,l+1),$$
    for all positive integers $k,c,l$, considered modulo $n$ in the first component and modulo $m$ in the second
    and third components. Hence $\mathcal{G}(X,r)$ is a transitive subgroup of $\Sym_X$ and thus $(X,r)$ is indecomposable.

    Let $(i,a,\mu),(j,c,\nu)\in X$ be elements such that $\sigma_{(i,a,\mu)}=\sigma_{(j,c,\nu)}$.
    Since
    $$(t^{a}(1),\mu,1-\delta_{i,t^{a}(1)})=\sigma_{(i,a,\mu)}(0,0,0)=\sigma_{(j,c,\nu)}(0,0,0)
    =(t^{c}(1),\nu,1-\delta_{j,t^{c}(1)}),$$
    it follows that $\nu=\mu$, $t^{a}(1)=t^{c}(1)$,
    and thus $t^{a}=t^{c}$. Since $t$ has order $m$, we have that
    $a=c$.
    In particular
    $$\sigma_{(i,a,\mu)}=\sigma_{(j,a,\mu)}.$$
    Hence
    $$(i,0,0)=\sigma_{(i,a,\mu)}(t^{-\mu}(i-t^{a}(1)), -\mu,0)=\sigma_{(j,a,\mu)}(t^{-\mu}(i-t^{a}(1)), -\mu,0)=(i,0,1-\delta_{j,i}),$$
    and thus $i=j$.
Therefore $(X,r)$ is irretractable.
\end{proof}

\begin{remark}
    {\rm Let $k$ be a positive integer and let $p_1,\dots ,p_k$ be distinct prime numbers. Let $n_1,\dots ,n_k$
    be positive integers such that $n_1\geq 4$. Let $n=p_1^{n_1-2}p_2^{n_2}\cdots p_k^{n_k}$ and $m=p_1$.
    Since $n_1\geq 4$, $\Aut(\Z/(n))$ has an element $t$ of order $p_1$. Hence, by Proposition \ref{newsolnew},
    there exists an indecomposable and irretractable solution of the YBE of cardinality $p_1^{n_1}\cdots p_k^{n_k}$.
}
\end{remark}

In \cite{CO22} it is shown that there exists simple solutions
$(X,r)$ of the YBE of square cardinality such that
$\mathcal{G}(X,r)$ is a simple left brace. Now we shall prove that
some of the solutions constructed in Proposition \ref{newsolnew}
are simple and  their permutation groups are simple left
braces.

\begin{theorem} \label{new-correction}
    Let $p,p_1,\dots ,p_k$,  ($k\geq 1$), be distinct prime numbers such that $p$ is a divisor of $p_i-1$
    for all $i=1,\dots ,k$. Let $m_1,\dots ,m_k$ be positive integers and let $n=p_1^{m_1}\cdots p_k^{m_k}$.
    Then there exists an invertible element $t\in\Z/(n)$ of order $p$ such that $t-1$ also is
    invertible in $\Z/(n)$.
    Let $X=\Z/(n)\times\Z/(p)\times\Z/(p)$ and let $(X,r)$ be the solution of the YBE described
    in Proposition \ref{newsolnew}, with $m=p$ and the automorphism of $\Z/(n)$ given by multiplication by $t$.
    In particular
    $$\sigma_{(i,a,\mu)}(j,c,\nu)=(t^{\mu}j+t^{a}, c+\mu,
    \nu+1-\delta_{i,t^{\mu}j+t^{a}}),$$
    for all $i,j\in\Z/(n)$ and all $a,c,\mu,\nu\in\Z/(p)$.
    Then $(X,r)$ is simple. Furthermore $\mathcal{G}(X,r)$ is a simple left brace.
\end{theorem}

\begin{proof}
        We shall construct a simple left brace $B$ with an orbit $X'$ by the action of the lambda
        map such that $B=\gr(X')_+$.
        Thus, by \cite[Theorem 5.1]{CO22}, $(X',r')$, where
        $$r'(x,y)=(\lambda_x(y),\lambda^{-1}_{\lambda_x(y)}(x)),$$
        for all $x,y\in X'$, is a simple solution of the YBE. Then we shall prove that the solution $(X',r')$
        is isomorphic to $(X,r)$ and $\mathcal{G}(X,r)\cong B$ as left braces.

        The following construction is a particular case of the construction of simple left braces
        in \cite[Section 5.3]{BCJO}.
        In particular, as in \cite[Section 5.3]{BCJO}, since
        $p\mid p_i-1$ and $m_i>0$, for all $i=1,\dots ,k$, one can prove that there exists an
        invertible element $t\in\Z/(n)$ of order $p$  such that $t-1$ also is
        invertible in $\Z/(n)$. Let $\mathbb{F}_p$ be the field of $p$ elements.
        Let $R=\mathbb{F}_p[x]/(x^{n-1}+\dots +x+1)$, and denote $\xi:=\overline{x}\in R$.
        Let $c\in \Aut(R,+)$ be defined by
        $$c(q)=\xi q,$$
        for all $q\in R$.  Let $f\in\Aut(R,+)$ be defined by
        $$f(q(\xi))=q(\xi^{t}),$$
        for all $q(x)\in \mathbb{F}_p[x]$.
        Let $b\colon R\times R\longrightarrow \Z/(p)$ be the unique bilinear form such that
        $b(\xi^{j},\xi^{i})=1-\delta_{j,i}$ for all $i,j\in\{ 0,1,\dots ,n-2\}$
        (note that $(1,\xi,\dots ,\xi^{n-2})$ is a basis of the $\mathbb{F}_p$-vector space $R$).
        As in \cite[Section 5.3]{BCJO}, we see that
        \begin{itemize}
            \item[(1)] $b$ is non-degenerate,
            \item[(2)] $f$ and $c$ are in the orthogonal group of $b$ and satisfy
            $fc=c^{t}f$,
            \item[(3)] $c-\id$ is invertible.
        \end{itemize}
        Consider the trivial left braces $(R,+,+)$ and $(\Z/(n),+,+)$.
        Let $T=R\rtimes \Z/(n)$ be the semidirect product of the trivial left braces $R$ and
        $\Z/(n)$ via the action
        $$\beta \colon \Z/(n)\longrightarrow \Aut(R,+)$$
        defined by $\beta_{a}(u)=c^{a}(u)$, where $\beta(a)=\beta_a$, for all $a\in\Z/(n)$ and $u\in R$.

        We define the map $\alpha\colon (\Z/(p),+)\longrightarrow \Aut(T,+,\cdot)$ by
        $$\alpha_{\mu}(u,a)=(f^{\mu}(u),t^{\mu}a),$$
        where $\alpha(\mu)=\alpha_{\mu}$, for all $\mu\in\Z/(p)$, all $a\in\Z/(n)$ and all $u\in R$.

        Let $\bar{b} \colon T\times T\longrightarrow \Z/(p)$ be the symmetric bilinear form defined by
        $$\bar{b}((u,a),(v,a'))=b(u,v),$$
        for all $a,a'\in \Z/(n)$ and all $u,v\in R$.

        Let $B=T\rtimes_{\circ} \Z/(p)$ be the asymmetric product of $T$ by $\Z/(p)$ via $\bar{b}$ and $\alpha$.

        By \cite[Theorem 5.7]{BCJO}, $B$ is a simple left brace.

        Note that the lambda map of the left brace $B$ satisfies that
        $$\lambda_{(u,a,\mu)}(v,a',\mu')=(c^{a}f^{\mu}(v), t^{\mu}a',\mu'-b(u,c^{a}f^{\mu}(v))),$$
        for all $u,v\in R$, all $a,a'\in \Z/(n)$ and $\mu,\mu'\in\Z/(p)$.

        Let $X'=\{ (\xi^{i}, t^{j}, k) \mid i\in \Z/(n), \, j,k\in\Z/(p)\} \subseteq B$. Since the addition in $B$ is defined by
        $$(u,a,\mu)+(v,a',\mu')=(u+v,a+a',\mu+\mu'+b(u,v)),$$
        for all $u,v\in R$, all $a,a'\in\Z/(n)$ and all $\mu,\mu'\in \Z/(p)$, it is clear that
        $B=\langle X'\rangle_+$. We also have that
        \begin{eqnarray*}\lambda_{(u,a,\mu)}(\xi^{i},t^{j},k)
            &=&(c^{a}f^{\mu}(\xi^{i}), t^{\mu}t^{j},k-b(u,c^{a}f^{\mu}(\xi^{i})))\\
            &=&(\xi^{a+t^{\mu}i},t^{\mu+j}, k-b(u,c^{a}f^{\mu}(\xi^{i}))\in X'.
        \end{eqnarray*}
        Let $\sigma_{(\xi^{i},t^{j},k)}\in \Sym_{X'}$ be defined by
        $$\sigma_{(\xi^{i},t^{j},k)}(\xi^{i'},t^{j'},k')=\lambda_{(\xi^{i},t^{j},k)}(\xi^{i'},t^{j'},k')=
        (\xi^{t^{j}+t^{k}i'},t^{j+j'},k'+1-\delta_{i,t^{j}+t^{k}i'}),$$
        for all $i,i'\in\Z/(n)$ and all $j,j',k,k'\in \Z/(p)$. Let $r'\colon X'\times X'\longrightarrow X'\times X'$
        be the map defined by
        $$r'((\xi^{i},t^{j},k),(\xi^{i'},t^{j'},k'))=(\sigma_{(\xi^{i},t^{j},k)}(\xi^{i'},t^{j'},k'),
        \sigma^{-1}_{\sigma_{(\xi^{i},t^{j},k)}(\xi^{i'},t^{j'},k')}(\xi^{i},t^{j},k)),$$
        for all $i,i'\in\Z/(n)$ and all $j,j',k,k'\in \Z/(p)$.
        Then $(X',r')$ is a solution of the YBE, in fact it is a subsolution of the solution
        $(B,r_B)$ of the YBE associated to the left brace $B$.
        It is easy to check that the map $g\colon (X,r)\longrightarrow (X',r')$ defined by
        $g(i,a,\mu)=(\xi^{i},t^{a},\mu)$, for all $i\in\Z/(n)$ and all $a,\mu\in \Z/(p)$, is an isomorphism
        of solutions.
        Since $(X,r)$ is indecomposable, we have that $X'$ is an orbit of $B$ by the action of the lambda map.
        By \cite[Theorem 5.1]{CO22}, $(X',r')$ is a simple solution. Therefore $(X,r)$ is a simple solution.
        Since $B=\gr(X')_+$, as in the proof of Proposition \ref{Bcbis} one can prove that
        $$\mathcal{G}(X',r')\cong B/\soc(B)$$
        as left braces. Since $B$ is a simple non-trivial left brace, $\soc(B)=\{ 0\}$. Hence
        $$\mathcal{G}(X,r)\cong \mathcal{G}(X',r')\cong B$$
        as left braces.  This completes the proof.
\end{proof}

{\bf Acknowledgments.} The first author was partially supported by
the project PID2020-113047GB-I00/AEI/10.13039/501100011033
(Spain).

\end{document}